\documentclass[11pt]{amsart}

\usepackage{fullpage}
\usepackage{amsmath}
\usepackage{amsfonts}
\usepackage{amssymb}

\newtheorem{theorem}{Theorem}[section]
\newtheorem{proposition}[theorem]{Proposition}
\newtheorem{lemma}[theorem]{Lemma}
\newtheorem{corollary}[theorem]{Corollary}
 
\theoremstyle{definition}
\newtheorem{definition}[theorem]{Definition}

\theoremstyle{remark}
\newtheorem{remark}[theorem]{Remark}
 
\numberwithin{equation}{section}
 
\newcommand{\be}{\begin{equation}}
\newcommand{\ee}{\end{equation}}

\newcommand{\bbC}{{\mathbb C}}
\newcommand{\bbZ}{{\mathbb Z}}
\newcommand{\bbR}{{\mathbb R}}

\newcommand{\calA}{{\mathcal A}}

\newcommand{\calH}{{\mathcal H}}

\DeclareMathOperator{\tr}{Tr}

\newcommand{\norm}[1]{\lVert#1\rVert}

\newcommand{\h}{\hbar}
\newcommand{\zbar}{\overline{z}}

\newcommand{\x}{\overline{x}}
\newcommand{\y}{\overline{y}}
\newcommand{\F}{\mathcal{F}\partial X_c}
\newcommand{\Fn}{\mathcal{F}\partial X_c^n}

\usepackage{color}
\input xy
\xyoption{all}
\usepackage{epsfig}

\begin{document}

\openup1pt
\newcommand{\medno}{\medskip\noindent}

\renewcommand{\H}{\hat{H}}
\newcommand{\Hcan}{H_{\text{\tiny can}}}

\bigskip
\bigskip

\title{Algebras of semiclassical pseudodifferential operators associated with Zoll-type domains in cotangent bundles}

\author{Gerardo Hern\'andez-Due\~nas$^*$}
\address{Department of Mathematics\\
University of Michigan\\ 530 Church Street \\ Ann Arbor, Michigan 48109-1043.}
\email{gerahdez@umich.edu}
\thanks{$^*$Work supported in part by CONACYT \#160147}

\author{Alejandro Uribe$^{**}$}
\address{Department of Mathematics\\
University of Michigan\\ 530 Church Street \\ Ann Arbor, Michigan 48109-1043.}
\email{uribe@umich.edu}
\thanks{$^{**}$Work supported in part by NSF grant DMS-0805878}

\maketitle

%\centerline{10/31/2009}, 4:00 PM

\begin{abstract}
%% Text of abstract
We are consider domains in cotangent bundles with the property that the null foliation of their boundary is fibrating and the leaves satisfy a Bohr-Sommerfeld condition (for example, the unit disk bundle of a Zoll metric).  Given such a domain, we construct an algebra of associated
semiclassical pseudodifferential operators with singular symbols.
The Schwartz kernels of the operators have frequency set contained in the union of the diagonal and the flow-out of the null foliation of the boundary of the domain. We develop a symbolic calculus, prove
the existence of projectors (under a mild additional assumption) whose range can be thought of as quantizing the domain, give a symbolic proof of a Szeg\"o limit theorem, and study associated propagators. 
\end{abstract}

\tableofcontents

\section{Introduction}

Let $M$ be a smooth compact manifold, and $X_c\subset X:= T^*M$ a closed compact domain with a smooth boundary.
In this paper we address the question: what quantum objects are naturally associated with $X_c$?  This 
question has been addressed indirectly in ``classical" cases.  For example,  if $X_c$ is the unit disk bundle
 associated to a Riemannian metric on $M$ assumed compact, then 
morally speaking to $X_c$ corresponds the ($\h$-dependent) subspace $\calH$ of $L^2(M)$ spanned by the 
eigenfunctions of the semiclassical Laplacian, $\widehat P = \h^2\Delta$, with eigenvalues in $[0,1]$.  This 
heuristics is in good measure justified by the theorem that, if $\Pi: L^2(M)\to \calH$ is the orthogonal 
projection, then for any $\h$-pseudodifferential operator $\widehat Q$ one has:
\begin{equation}\label{intro1}
\lim_{\h\to 0} \frac{1}{\dim \calH}\ \tr \Pi \widehat Q\Pi = \frac{1}{\text{Vol} X_c}\ \int_{X_c}\,Q \, d\lambda ,
\end{equation}
where $Q$ is the principal symbol of $\widehat Q$ and $d\lambda$ is Liouville measure (see \cite{Gui77}).  
But can one say more?

We discuss a more systematic answer to this question in the case when $X_c$ is ``fibrating and Bohr-
Sommerfeld", by which we mean the following.  The boundary $\partial X_c$ is always foliated by curves 
tangent to the kernel of the pull-back of the symplectic form.  The fibrating condition is that there exists
a manifold $S$ and a submersion $\pi: \partial X_c\to S$ whose fibers are the leaves of the null-foliation.
This is satisfied iff $X_c$ has a globally defining function whose Hamilton flow on $\partial X_c$
is periodic with a common minimal period.  The Bohr-Sommerfeld condition is that all leaves, $\gamma$, of
$\pi$ satisfy
\[
\int_\gamma \alpha \in 2\pi\bbZ,
\]
where $\alpha$ is the tautological one-form in $T^*M$.  
We will show that under these conditions there exist spaces $J^{\ell, m}$ of pseudodifferential operators with singular symbols naturally associated with $X_c$.  (Here $(\ell,m)$ is a bi-degree, 
to be explained later).  The frequency sets of their Schwartz kernels are contained in the union
of Lagrangian submanifolds of $T^*(M\times M)$
\[
\Delta_c' \cup \F',
\]
where
\[
\Delta_c = \{ (\x, \x)\in X_c\times X_c \}
\]
and $\F$ is the flow-out
\[
\F = \{ (\x, \y)\in \partial X_c\times \partial X_c\;;\; \x,\ \y\ \text{in the same leaf}\ \}.
\]
Intuitively speaking, the diagonal part, $\Delta_c$, is expected to be a part of any pseudodifferential operator calculus associated with $X_c$.  The flow-out part, $\F$, is there because the fibrating-and-Bohr-Sommerfeld conditions imply that there should be a significant part of Hilbert space associated with the symplectic reduction of the boundary of $X_c$. The two symbols, one on $\Delta_c$ and one on $\F$, have a compatibility condition that comes about most naturally in our setting.

In particular, the space $J^{-1/2, 1/2}$ is closed under composition, and it includes the projector $\Pi$ mentioned above, in the case of Zoll metrics.    The Bohr-Sommerfeld condition is needed for the existence of a global symbolic calculus, and goes along with having to restricting Planck's constant to take the values $\h = 1/N$, $N=1,2,\ldots $.
The fibrating condition is needed in order for $\F$ to be a closed submanifold of $X\times X$.

The Schwartz kernels are semiclassical analogues of the oscillatory integrals with singular symbols
of \cite{MU} and \cite{GU}, associated to a pair of intersecting conic Lagrangian submanifolds. See also \cite{Jo}, where a precise calculus for a more generalized class is discussed, and \cite{HaVa} where a connection with a class of Legendre distributions is explained.  In the conic case, the realization that if the Lagrangians are the diagonal and a flow-out one obtains a symbol calculus is due to the results in \cite{AUh}.  This is possible because
\[
\Delta_c\circ\F,\ \F\circ\Delta_c,\ \F \circ \F \subset\F\ ; \quad \Delta_c \circ\Delta_c\subset \Delta_c,
\]
so composing two operators with wave-front set in $\Delta_c \cup \F$ produces an operator
with wave-front set contained in the same union.  We believe that the present semiclassical setting provides a very natural expression for the Antoniano-Uhlmann algebra.

In the Zoll case, there have been numerous papers aimed at refining the general result (\ref{intro1}), mostly in terms of the reminder estimate.   For instance, in \cite{GuiOki3}, the second term in the Szeg\"o formula for Zoll manifolds is proved.  Other references in this direction are \cite{GuiOki2,Di,La1,La2,Wi1}. 

Here we are not focusing on remainder estimates, but on the fact that there is an operator algebra with a symbolic calculus that provides a quantization of $X_c$ and, among other things, a broader (symbolic) setting for Szeg\"o limit theorems. Furthermore, the existence of the operator algebra allows us to go farther in the analysis of the operators $\Pi\widehat Q\Pi$ mentioned above, with respect to previous works.

The issues we raise here are also connected with work of \cite{GuLe}, 
on the relationship between symplectic cutting and quantization, in the homogeneous (non-semiclassical) category.  They consider the case when $X_c = \phi^{-1}(-\infty, 0]$ where $\phi: X\to\bbR$ is the moment map for a homogeneous action of the circle group on $S^1$.   In this setting one can form the symplectic cut
\[
Y = X_c/\sim,
\]
Their work centers on the algebra of operators $\{\Pi\widehat Q \Pi \;;\; [\Pi, \widehat Q] = 0\}$ where
$\widehat Q$ ranges over (non-semiclassical) $\Psi$DOs on $M$ and $\Pi$ is a spectral projector, as above, 
associated to a quantization of the circle action by a Fourier integral operator.  Roughly speaking they show 
that such an algebra can be considered a quantization of the symplectic manifold $Y$.

\subsection{Main results}   We now describe our main results.

\begin{theorem}\label{Main}  Let $X_{c}\subset T^{*}M$ be as above ($\partial X_{c}$ ``fibrating and Bohr-Sommerfeld'').  
Then there exist vector spaces of semiclassical pseudodifferential operators with singular symbols,
$J^{\ell, m}( M\times M ;\Delta,\F)$, where $\h$ is restricted to the sequence 
$1/N, N=1,2,\ldots$, such that:
\begin{enumerate}
\item The frequency set of the Schwartz kernel of any operator in the algebra is contained in the union
$\Delta'\cup\F '$.
\item Let $\Sigma =\Delta\cap \F$. Then $J^{\ell,m}$ are microlocally Lagrangian states of order ${\ell+m}$ on ${\Delta' \setminus \Sigma'}$ and ${\ell}$ on $\F' \setminus \Sigma'$, and therefore there are symbol maps:
\[
\begin{array}{ll}
\sigma_0: J\to |\wedge|^{1/2}(\Delta\setminus\Sigma), & \sigma_1: J\to |\wedge|^{1/2}(\F\setminus\Sigma)
\end{array}
\]
(where $|\wedge|^{1/2}$ denotes the space of half-densities, and we will ignore Maslov factors).
There is in fact a symbolic calculus, that will be described below.
\item $J^{\ell,m}\circ J^{\tilde\ell,\tilde m}\subset J^{{\ell+\tilde\ell+1/2},{m+\tilde m-1/2}}$, in particular
$J^{-1/2,1/2}$ is an algebra. 
\item Assume that $X_c$ is compact.  Then every 
\[
\widehat{A}\in J^{-1/2,1/2}( M\times M;\Delta,\F)
\]
with microsupport contained in $X_c$ is smoothing and
\[
\tr(\widehat{A}) = (2\pi)^{{-n}}\ \hbar^{-n}\,\int_{X_c} \sigma_0(\widehat{A})\, \frac{\omega^n}{n!} + O(\hbar^{-n+1}\log(1/\h))
\]
where $\omega$ is the symplectic form of $T^*M$ and $n$ is the dimension of ${M}$.
\end{enumerate}
\end{theorem}

The manifold $X_{c}$ has an associated operator algebra, $\mathcal{A}_{X_c}$, which consists of elements in the algebra 
$J\left( M\times M; \Delta, \partial X_c \right)$ which are microlocally of order O($\h^\infty$) in the complement $T^*M\setminus X_c$.

%\begin{theorem}\label{Main2}
%Assume $X_c=P^{-1}[E_1,E_2]$, 
%and 
%$P: T^*M\to\bbR$ satisfies the hypothesis of Section \ref{sec:Projector}. Then there exist orthogonal projections $\Pi\in J^{-1/2,1/2}$.  Moreover,
%for any zeroth-order pseudodifferential operator on $M$, $\widehat{Q}$, the ``cut" operator 
%$ \Pi \widehat{Q} \Pi$
%is in the algebra, and $\sigma_0$ can be identified with $Q|_{X_c}$.
%\end{theorem}

\begin{theorem}\label{Main2}  (See \S \ref{ProjectorExistence} for details.)
Assume $\partial X_c$ is compact and of contact type.
Then there exist orthogonal projections $\Pi\in J^{-1/2,1/2}$ whose symbol $\sigma_{0}$ is the characteristic
function of $X_{c}$.  Moreover,
for any zeroth-order pseudodifferential operator on $M$, $\widehat{Q}$, the ``cut" operator 
$ \Pi \widehat{Q} \Pi$
is in the algebra, and $\sigma_0$ can be identified with $Q|_{X_c}$.
\end{theorem}

\medskip
As an immediate corollary we obtain the following Szeg\"o limit theorem:
\begin{corollary}\label{Szgo}
Assume that $X_c$ is compact, $\Pi\in J^{-1/2,1/2}$ an orthogonal projector as in the previous theorem 
and let $\h = 1/N$.  Then for any integer $m\geq 0$
\[
\tr(\Pi\widehat{Q} \Pi)^m = (2\pi)^{{-n}}\ N^{n}\,\int_{X_c} Q^m\, \frac{\omega^n}{n!} + O(N^{n-1}\log(N))
\]
\end{corollary}

Finally, we have a result on the propagator $\Pi e^{-it\h^{-1} \Pi \widehat Q \Pi}$ where 
the symbol, $Q$, of the pseudodifferential operator $\widehat Q$ preserves $X_c$ ``to second
order":

\begin{theorem}
\label{Main3}
Suppose $\widehat Q$ is a zeroth-order  semiclassical pseudodifferential operator satisfying the conditions of Lemma \ref{Lemma:QHypothesis}. Assume $sub \widehat Q(\h)=0$. Then
\[
\Pi e^{-it\h^{-1} \Pi \widehat Q \Pi}  \in J^{-1/2,1/2}\left( M\times M ; \Delta(t), \F (t) \right),
\]
where 
\[
\Delta(t)=\left\{ \left( \x,\y \right) \big | \x, \y \in T^{*}(M\times M), \x=\Phi_t^Q (\y) \right\}
\]
\[
\F(t)=\left\{ \left( \x,\y \right) \big | \x,\y \in \partial X_c,~ \exists s \in \mathbb{R} \textrm{ such that } \x=\Phi_s^P \Phi_t^Q (\y) \right\}.
\]
\end{theorem}

\bigskip\noindent
{\bf Remarks:}  
\begin{list}{\labelitemi}{\leftmargin=1em}
\item[(1)] The result on the trace (part (4) of the Theorem \ref{Main}) extends to the spaces $J^{\ell,m}$ provided $m \geq -1/2$.  If $m$ is sufficiently negative, the leading contribution to the trace comes from the asymptotic singularity of the kernel of the operator at $\Sigma$.
\item[(2)] The symbol calculus for $\sigma_0$ is just the usual pseudodifferential calculus.  The symbol calculus for $\sigma_1$ is more complicated (and non-commutative).  In fact $\sigma_1$ comes with an extension to a distribution on $\F$ conormal to $\Sigma$, and there exists a formula for the smooth part of the $\sigma_1$ of the composition in terms of the corresponding extensions of the factors.  
\item[(3)]  The restriction that $\h=1/N$ is necessary to have a well-defined global symbol on the flow-out $\F$, see \cite{GeoA} chapter VII, \S 0.  Locally elements of the $J^{\ell, m}$ are given by semiclassical oscillatory integrals with $\h\to 0$ continuously.
\item[(4)] The error estimate $O(\h\log(1/\h)$) is sharp for the class $J^{-1/2,1/2}$, but in Corollary \ref{Szgo} the error should be $O(\h^2)$, see Remark \ref{NotaError}.
\item[(5)] The idea that the image of $\Pi$ quantizes $X_c$ and that the operators $\Pi \widehat{Q} \Pi$ can be considered as associated observables appeared first in the physics literature, see \cite{BP1} \S II E and \cite{BP2} (where a connection with symplectic cutting is also made).  The operators
$\Pi \widehat{Q} \Pi$ do not form an algebra, however, while the class $J^{-1/2,1/2}$ does.
\item[(6)] Our work also implies that the results of \cite{GuLe} 
also hold in the semiclassical case. 
\end{list}

\medskip
The paper is organized as follows. 
In Section \ref{sec:GuiUhlOperators}, the spaces $J^{\ell,m}$ are defined in a model case by oscillatory 
integrals with amplitudes having expansions in $\h$ with coefficients that are classical symbols. In Section
\ref{sec:GeneralCase}, the spaces are defined globally on manifolds, and the existence of a symbolic 
calculus is established. The principal symbol on each Lagrangian blows-up at the intersection, and they 
satisfy a \emph{symbolic} compatibility condition there. The theorem on the trace is proved in Section \ref
{sec:AsympTrace}. Section \ref{sec:Projector} considers cases where the algebra admits projectors,  
yielding a symbolic proof of a generalized Szeg\"o limit Theorem. Section \ref{sec:Propagator} studies the
propagator of certain elements in the algebra, and proves an Egorov-type theorem. Finally, 
Section \ref{sec:SCNA} shows numerically a surprising phenomenon of propagation of coherent states
in situations not considered in Section \ref{sec:Propagator}.

%-------section----------

%-------section-------

\section{The model case}

\label{sec:GuiUhlOperators}

\subsection{Definitions}
\label{sec:ModelCase}

We begin by discussing the microlocal model case: $M=\bbR^n$ and 
$X_c^n=\left\{ p_1 \ge 0 \right\}$, where $(x_1,\ldots ,x_n, p_1,\ldots, p_n)$ are canonical
coordinates in $T^*\bbR^n$.   Let $\Fn$ be the flow-out of $\partial X_c^n=\left\{ p_1=0\right\}$. We will use this case to define
operators $J^{\ell,m}(\bbR^n \times \bbR^n;\Delta^n,\Fn)$, where $\Delta^n\subset T^*\bbR^n\times T^*\bbR^n$ is the diagonal.

Roughly speaking, elements in $J^{\ell,m}(\bbR^n\times \bbR^n, \Delta^n, \Fn)$ will be defined as oscillatory integrals with amplitudes as follows:

\begin{enumerate}
\item We denote by $\mathcal{A}^{\ell,m}$ the class of all smooth functions $a(s,x,y,p,\sigma,\h)$ with compact support in $s,x,y,p$ such that, as $\h\to 0$ 
\[
a(s,x,y,p,\sigma,\h)\sim\sum_{j=-\ell}^\infty \h^{j} a_j (s,x,y,p,\sigma)
\]
(in a sense that will be explained below) where, for each 
$j$, $a_j (s,x,y,p,\sigma)$ is a polyhomogeneous classical symbol in $\sigma$ of degree $m$:
\[
a_j(s,x,y,p,\sigma) \sim \sum_{r=m}^{-\infty} a_{j,r}(s,x,y,p,\sigma),\quad \forall\lambda>0 \ 
a_{j,r}(s,x,y,p,\lambda\sigma) = \lambda^r a_{j,r}(s,x,y,p,\sigma).
\]
\item The operators in $J^{\ell,m}(\bbR^n\times \bbR^n; \Delta^n, \Fn)$ are those whose Schwartz kernels are of the form
\begin{equation}
\label{OscInt}
A(x,y,\hbar) = \frac{1}{(2\pi \hbar)^n}\int e^{\frac{i}{\h}\Bigl[(x_1-y_1-s)p_1 + (x'-y')p'\Bigr] + is\sigma}\ a(s,x,y,p,\sigma ,\h)\ ds\,dp\,d\sigma ,
\end{equation}
where we have split the variables: $x=(x_1,x'),\ y=(y_1,y') \ (x', y'\in\bbR^{n-1})$.
\end{enumerate}

We now give the details.
\begin{definition}
Let $(x,y)$ and $s$ be the standard coordinates on $\mathbb{R}^{2n}$ and $\mathbb{R}$ respectively, and let $p$, $\sigma$ be the dual coordinates to $x$, $s$. We denote by $z=(x,y,s)$ coordinates in 
$\mathbb{R}^{2n} \times \mathbb{R}$.  Define $\calA^{\ell, m}$ to be the space of smooth families 
$a(s,x,y,p,\sigma,\hbar)$ compactly supported in $z$, $p$ such that
\begin{equation}
\label{AmplitudeIneq}
\left| \left( \partial / \partial z \right)^{\alpha} \left( \partial / \partial p \right)^{\beta} \left( \partial / \partial \sigma \right)^{\gamma} a \right| \le C \hbar^{-\ell} (1+|\sigma |)^{m -|\gamma|},
\end{equation}
for some constant $C=C(\alpha,\beta,\gamma)$ and $\h \in (0,h_0 ]$ for some fixed $h_0 > 0$.
\end{definition}

Let $a_{j}(s,x,y,p,\sigma)\in S^{m}$ be a sequence of classical symbols in the $\sigma$ variable, of order $m$, and compactly supported in $(s,x,y,p)$, i.e., $a_j$ is smooth,
\[
\left| \left( \partial / \partial z \right)^{\alpha} \left( \partial / \partial p \right)^{\beta} \left( \partial / \partial \sigma \right)^{\gamma} a_{j}(s,x,y,p,\sigma) \right| \le C(\alpha,\beta,\gamma)\left( 1+|\sigma |\right)^{m-|\gamma|},
\]
for some constant $C=C(\alpha,\beta,\gamma)$, and there exists a sequence $a_{j,r}$ of smooth functions 
in 
$\sigma\neq 0$, homogeneous of degree $r$ in $\sigma$, such that $a_j \sim \sum_{r=m}^{-\infty} a_{j,r}$ 
in the standard sense.  Given $a\in\calA^{\ell,m}$,
we will say that $a\sim \sum_{j=-\ell}^{\infty}\hbar^{j}a_j$ iff for all integers $K\geq 0$
\[
a-\sum_{j=-\ell}^{-\ell+K}\hbar^{j} a_j  \in \calA^{\ell-K-1,m}.
\]

\begin{definition}
\label{def:AlgModelCase}
Denote by $\calA_{\text{classical}}^{\ell , m}\subset\calA^{\ell, m}$ the set of
all $a(s,x,y,p,\sigma,\hbar)$ that satisfy $a\sim \sum_{j=-\ell}^{\infty}\hbar^{j}a_j$, as above, and let $\Sigma^n=\Delta^n \cap \Fn$.
Define $J^{\ell,m}\left( \mathbb{R}^n\times \bbR^n;\Delta^n,\Fn \right)$ to be the set of kernels 
of the form
\[
A(x,y,\h) + F_1(x,y,\h)+F_2(x,y,\h)
\]
where $A(x,y,\h)$ is as in \eqref{OscInt} with $a\in \calA_{\text{classical}}^{\ell',m'}$ where 
\[
\ell'=\ell+1/2,\quad m'=m-1/2,
\]
and $F_1$, $F_2$ are the kernels of semiclassical Fourier integral operator in \\
$sc-I^{\ell+m}\left( \bbR^n \times \bbR^n ; \Delta^n \setminus \Sigma^n \right)$, $sc-I^{\ell} \left( \bbR^n \times \bbR^n ; \Fn \right)$, respectively.
\end{definition}

\begin{remark}
As we will see, it
is necessary that $F_1$ and $F_2$ appear in the definition in order for the classes $J$ to be closed
under composition.
\end{remark}

We have yet to give sense to the formula in equation \eqref{OscInt}. Denote by $D=1+D_{s}^2$ where $D_s=\frac{1}{i}\frac{\partial}{\partial s}$. Notice that $D e^{is\sigma}=(1+\sigma^2)e^{is\sigma}$. For $a\in A^{\ell' ,m'}_{\text{classical}}$, we define $A(x,y,\h)$ as
\begin{equation}\label{sense}
A(x,y,\h)=\frac{1}{(2\pi \h)^n} \int e^{ i \h^{-1} (x-y) p+is\sigma}\frac{ 1 }{(1+\sigma^2)^k}D^{k}\left( e^{-i\h^{-1}sp_1} a(s,x,y,p,\sigma,\h) \right) ds dp d\sigma .
\end{equation}
The integral above is absolutely convergent for $k >> 0$ large enough since 
$D^{k}\left( e^{-i\h^{-1}sp_1} a \right)$ is $O(\sigma^m)$.  
It can be easily checked that the definition above does not depend on $k$ using integration by parts.

Equivalently, one can define the integral in \eqref{OscInt} as an iterated integral, where one first 
integrates over the variables $(s,p)$ 
with respect to which the amplitude is compactly supported.  The resulting function of $\sigma$ is
rapidly decreasing and therefore integrable.  Both of these interpretations are useful in proofs.

\medskip

We now state the first property of the kernels we have just defined:

\begin{proposition}\label{FSetModelCase}
For any $A \in J^{\ell,m}\left( \mathbb{R}^{n}\times \bbR^n;\Delta^n,\Fn \right)$, the frequency set of $A$  is contained in the union of the Lagrangians $\Delta^n$ and $\Fn$:
\[
FS\left( A \right) \subset {\Delta^n}' \cup {\Fn}'.
\]
Moreover, away from the intersection $\Sigma=\Delta^n \cap \Fn$, $A$ is microlocally  in the space $\text{sc-} I^{\ell+m}\left( \bbR^n\times \bbR^n ; \Delta^n \setminus \Sigma^n \right)$ and $\text{sc-}I^{\ell}\left( \bbR^n \times \bbR^n ; \Fn \setminus \Sigma^n \right)$, where $\text{sc-}I$ denote the spaces of semiclassical Fourier integral operators.
\end{proposition}

\begin{remark}
This is a consequence of Theorem \ref{SymbCalculus}, proved in Section \ref{hTransform}.
\end{remark}

%----subsection------

\subsection{The symbol maps}

By the second part of Propositon \ref{FSetModelCase}, one has two symbol maps:
\begin{equation}
\begin{array}{ccl}
 &\sigma_0\nearrow  & |\wedge|^{1/2}(\Delta^n \setminus\Sigma)\\
 J^{\ell,m}\left( \mathbb{R}^{n} \times \bbR^n ;\Delta^n,\Fn \right) & \\
 & \sigma_1\searrow &  |\wedge|^{1/2}(\Fn \setminus\Sigma).
\end{array}
\end{equation}
It is easy to see that, for $A$ as above, they are given by the following formulae:

\begin{equation}\label{symbolsModelCase}
\begin{array}{crl}
\sigma_0(A) := 2\pi ~a_{-\ell',m'}(s,x,y,p,\sigma)\sqrt{dx dp}_{\big{|}_{y=x,s=0,p_1=\sigma}} \quad\text{and}\\ \\
\sigma_1(A) :=  \sqrt{2\pi} ~\int a_{-\ell'}(s,x,y,p,\sigma)e^{is\sigma}d\sigma \sqrt{dxdy_1 dp'}_{\big{|}_{y'=x',p_1=0,s=x_1-y_1}}.

\end{array}
\end{equation}
Notice that $\sigma_0$ has a singularity as $\sigma = p_1$ converges to zero, that is, as the point where $\sigma_0$ is evaluated tends to the intersection, $\Sigma$.  The same is true of $\sigma_1$, and the leading singularities of $\sigma_0$ and $\sigma_1$ are Fourier transforms of each other.  This is exactly as in 
\S 5 of \cite{GU}  , and (appropriately understood) will be true in the manifold case as well.

\begin{proposition}\label{ExactSeqModelCase}  One has the following exact sequence:
\[
0\to J^{\ell, m-1}\oplus J^{\ell-1, m} \to J^{\ell,m} \stackrel{\sigma_0}{\to} R^{m}(\Delta^n\setminus\Sigma)\to 0.
\]
\end{proposition}
Here $R^{m}(\Delta^n\setminus\Sigma)$ is, roughly speaking, the space of smooth functions on 
$\Delta^n\setminus\Sigma$ that have singularities of degree $m$ at $\Sigma$ in the normal directions (for details we refer to \cite{GU}, \S 5 and \S 6).

\medskip
The classes of operators with kernels in the $J^{\ell, m}$ are closed under composition in the following sense:

\begin{proposition}\label{CompoModelCase}
The composition of properly supported operators with kernels in $J^{\ell, m}$ and $J^{\tilde\ell, \tilde m}$, respectively, is an operator with kernel in $J^{\ell+\tilde\ell+1/2, m+\tilde m-1/2}$.  In particular $J^{-1/2, 1/2}$ is an algebra:
\[
J^{-1/2, 1/2}\circ J^{-1/2, 1/2}\subset J^{-1/2, 1/2},
\]
and for any $u\in J^{\ell,m}$, $v\in J^{\tilde \ell,\tilde m}$ , we obtain (ignoring Maslov factors and half-densities):
\[
\sigma_0 (u \circ v)(\bar x, \bar x)=\sigma_0 (u) (\bar x,\bar x) \sigma_0 (\bar x, \bar x), ~~\sigma_1(\bar x, \bar y)=\frac{1}{\sqrt{2\pi}} \int \sigma_1 (u)(\bar x, \bar z(t)) \sigma_1 (\bar z(t), \bar y) d t,
\]
where $\bar z(t)$ is the bicharacteristic curve joining $\bar x$ and $\bar y$.
\end{proposition}

Note that, by the symbol calculus for semiclassical FIOs (\cite{GuS}) there are formulae for the symbols $\sigma_0$, $\sigma_1$, of the composition.  (Microlocally near $\Delta \setminus\Sigma$ the operators are pseudodifferential and the symbol calculus for $\sigma_0$ is the usual one). We will have  more to say about the symbol calculus in the next section.

\begin{proof}  (Our proof follows the lines of the proof of Theorem 0.1 in \cite{AUh}.)
It is not hard to show that the classes $J$ are closed under composition by operators
$F_1$, $F_2$ as in definition \ref{def:AlgModelCase}.  Therefore we start with kernels as in \eqref{OscInt}.
It is also not hard to show that it suffices to prove the theorem in the case when the amplitudes 
of these kernels are independent of $\h$, in which case $\ell = \tilde \ell =-1/2$.  Let us therefore consider
\[
u(x,y)=\frac{1}{(2\pi\h)^n}\int e^{i\h^{-1}\left( (x-y)p-sp_1\right)+is\sigma}a(s,x,y,p,\sigma)dsdpd\sigma,
\]
\[
v(y,z)=\frac{1}{(2\pi \h)^{n}}\int e^{i\h^{-1}\left( (y-z)\tilde p-\tilde s \tilde p_1 \right)+i\tilde s \tilde \sigma}\tilde{a} (\tilde s, y,z,\tilde p, \tilde \sigma)d \tilde s d \tilde p d \tilde \sigma,
\]
where $a$ and $\tilde a$ are classical symbols in $\sigma$ and $\tilde \sigma$ of orders $m'=m-1/2$ and 
$\tilde m' =\tilde m-1/2$, respectively.  For $\chi_0$ a function that vanishes near the origin, and $1$ outside a compact support, we have
\[
u(x,y)=\frac{1}{(2\pi\h)^n}\int e^{i\h^{-1}\left( (x-y)p-sp_1\right)+is\sigma}a(s,x,y,p,\sigma) \chi_0 (\sigma) dsdpd\sigma + \tilde u,
\]
\[
v(y,z)=\frac{1}{(2\pi \h)^{n}}\int e^{i\h^{-1}\left( (y-z)\tilde p-\tilde s \tilde p_1 \right)+i\tilde s \tilde \sigma}\tilde{a} (\tilde s, y,z,\tilde p, \tilde \sigma) \chi_0(\tilde \sigma ) d \tilde s d \tilde p d \tilde \sigma+ \tilde v,
\]
where $\tilde u$, $\tilde v \in \text{sc-}I^{\ell} \left( \bbR^n \times \bbR^n ; \Fn \right)$. We can therefore
assume that $a$, $\tilde a$ are zero near $\sigma=0$, $\tilde \sigma=0$ respectively. The composition has Schwartz kernel:
\[
\omega(x,z)=\frac{1}{(2\pi \h)^n}\int e^{is\sigma +i\tilde s \tilde \sigma} D ds d\tilde s d\sigma d \tilde \sigma dp,
\]
with
\[
D=\frac{1}{(2\pi\h)^n} \int e^{i\h^{-1} \phi} \left(a\cdot \tilde a\right)\,
(s,\tilde s,x,y,z,p,\tilde p, \sigma, \tilde \sigma)\, dy 
d\tilde p,
\]
where
\[
\phi= (x-y)p+(y-z)\tilde p-\tilde s\tilde p_1-sp_1.
\]
The critical points for this phase for each $x,z,p$ fixed are: $\tilde p=p$, $y'=z'$, $y_1=z_1-\tilde s$. So, by the stationary phase theorem, we obtain:
\[
D\sim e^{i\h^{-1}\phi(y=(z_1-\tilde s,z'),\tilde p=p)}a\Bigl(s,x,y=(z_1-\tilde s,z'),p,\sigma\Bigr)\ 
 \tilde a\Bigl(\tilde s, y=(z_1-\tilde s,z'),z, p, \tilde \sigma\Bigr)
\]
as $\h\to 0$.
Stationary phase in fact gives us a complete asymptotic expansion in increasing
powers of $\h$, with
coefficients derivatives of $a$ and $\tilde a$ evaluated at the critical points. It suffices to consider
the contribution to $\omega$ of the leading term in the expansion of $D$ (the other terms are treated
in the same manner):
\[
\begin{aligned}
\omega_0(x,z) := 
\frac{1}{(2\pi \h)^n}\int & e^{i\h^{-1} \left[ (x'-z')p'+(x_1-z_1-s-\tilde s)p_1 \right]+is\sigma +i \tilde s \tilde \sigma}   \\
& a(s,x,(z_1-\tilde s,z'),p,\sigma)\, \tilde  a (\tilde s,(z_1-\tilde s,z'),z,p,\sigma)\, ds d\tilde s dp d\sigma d\tilde \sigma . 
\end{aligned}
\]
Making the change of variables $t=s+\tilde s$, we can write 
\[
\omega_0(x,z) = \frac{1}{(2\pi \h)^n} \int e^{i\h^{-1}\left[ (x_1-z_1-t)p_1+(x'-z')p' \right]+i t \tilde \sigma } b(t,x,z,p,\tilde \sigma) dt dp d\tilde \sigma,
\]
where
\[
b(t,x,z,p,\tilde \sigma)=\int e^{is (\sigma-\tilde \sigma)}a(s,x,(z_1-t+s,z'),p,\sigma)\,\tilde a(t-s,(z_1-t+s,z'),z,p,\tilde \sigma)\,ds d\sigma .
\]
Next, we split $\omega_0$ in three parts,  as follows.

Let $\chi_k=\chi_k(\sigma, \tilde \sigma)$, $k=1,2$ be  smooth, classical symbols 
of degree zero such that 
\[
\chi_1= \left\{ 
\begin{array}{llrr}
1 \text{ for }| \sigma | \le \frac{1}{2} \epsilon |\tilde \sigma| \\ \\
0 \text{ for } |\sigma | \ge \epsilon |\tilde \sigma |
\end{array}
\right. ~~\text{ and }~~
\chi_2= \left\{ 
\begin{array}{llrr}
1 \text{ for }| \tilde \sigma | \le \frac{1}{2} \epsilon | \sigma| \\ \\
0 \text{ for } | \tilde \sigma | \ge \epsilon | \sigma |
\end{array}
\right.
\]
for $\epsilon << 1$.   We let
\[
\Upsilon_k (x,z) = 
\frac{1}{(2\pi \h)^n} \int e^{i\h^{-1}\left[ (x_1-z_1-t)p_1+(x'-z')p' \right]+i t \tilde \sigma } b(t,x,z,p,\tilde \sigma) \,\chi_k\, dt dp d\tilde \sigma\quad  k=1,2,
\]
and
\[
\Upsilon_0 = \omega_0-\Upsilon_1-\Upsilon_2.
\]
We will show that $\Upsilon_k$, $k=1,2$ is a semiclassical state in the flow out
while $\Upsilon_0$ is an integral of the form \eqref{OscInt}.

Let us rewrite, for $k=1,2$,
\[
\Upsilon_k = \frac{1}{\left( 2\pi \hbar \right)^{n}} 
\int e^{\h^{-1}(x'-z')p'} \left(\int e^{i\h^{-1}(x_1-z_1-t)p_1} 
C_k(x,z,t,p)\, dp_1\,dt\right)\, dp'
\]
where
\[
C_k(x,z,t,p) :=\int e^{it\tilde \sigma+is(\sigma-\tilde \sigma)}\chi_k a\; \tilde a \, 
ds d\sigma d\tilde \sigma, 
\]
interpreted as an iterated integral (first with respect to $s$).
On the support of $\chi_1$ one has
$|\sigma| \le \epsilon |\tilde \sigma|$ which implies
$|\sigma-\tilde \sigma| \ge |\tilde \sigma|-|\sigma| \ge (1-\epsilon)|\tilde\sigma|$, and therefore
\[
\frac{1}{|\sigma-\tilde \sigma|^N} \le \frac{1}{\left(1-\epsilon\right)^N} \frac{1}{|\tilde\sigma|^N}
\le \left( \frac{\epsilon}{1-\epsilon}\right)^{N} \frac{1}{|\sigma|^{N}} 
\]
for all $N>0$.
Since $a ~\tilde a$ vanish near $\tilde \sigma=0 =\sigma$ and $\chi_1$ vanishes in a conic neighborhood
of the diagonal, we can integrate by parts repeatedly to obtain
\[
C_1(x,z,t,p)=\int e^{it\tilde \sigma+is(\sigma-\tilde \sigma)} \frac{\chi_1}{(-i)^{N} 
(\sigma-\tilde \sigma)^N}D_{s}^{N} (a\tilde a) ds d\sigma d\tilde \sigma.
\]
Since $D_{s}^{N}(a\tilde a)$ is of the same order in $\sigma,\tilde \sigma$ and $N$ is arbitrary, $C_1$ is Schwartz in the variable $t$.  Applying once again the method of stationary phase this implies that
\[
\int e^{i\h^{-1}(x_1-z_1-t)p_1} C_1(x,z,t,p)\, dp_1\,dt
\]
is a semiclassical symbol and therefore $\Upsilon_1$ is a semiclassical state on the flow out.
The proof for $\Upsilon_2$ is analogous. 

To show that $\Upsilon_0$ is as desired we only need to show that
\[
\begin{aligned}
b_0(t,x,z,p,\tilde \sigma)=\int & e^{is (\sigma-\tilde \sigma)}  a(s,x,(z_1-t+s,z'),p,\sigma)\, \\
 & \tilde a(t-s,(z_1-t+s,z'),z,p,\tilde \sigma)\,\left(1-\chi_1(\sigma,\tilde\sigma)-\chi_2(\sigma,\tilde\sigma) \right)\, ds d\sigma.
\end{aligned}
\]
 is a classical symbol in $\tilde \sigma$ of order $m'+\tilde m'$.

Making the change of variables $\tau = \frac{\sigma}{\tilde \sigma}$, we obtain: 
\[
\begin{aligned}
b_0=\tilde \sigma \int e^{i s \tilde (\tau-1)} & a(s,x,(z_1-t+s,z'),p,\tilde \sigma \tau) \\
& \tilde a(t-s,(z_1-t+s,z'),z,p,\tilde \sigma)(1-\chi_1(\tilde \sigma \tau, \tilde \sigma)-\chi_2(\tilde \sigma \tau, \tilde \sigma)) ds d\tau.
\end{aligned}
\]
Notice that $\tau$ is bounded in the support of the amplitude, and so using stationary phase in $\tilde \sigma$, we get an expansion for $b_0$ in terms of $\tilde \sigma$, where the leading term is
\[
b\sim 2\pi a_{m'}(0,x,(z_1-t,z'),p,\tilde \sigma) ~ \tilde a_{\tilde m'}(t, (z_1-t,z'),z,p,\tilde \sigma),
\]
and we conclude $b$ is a classical symbol in $\tilde \sigma$ of degree $m'+\tilde m'$.

\medskip
The symbol of $\omega(x,z)$ in the diagonal is easily computed as:
\[
\sigma_{0}(\omega)(x,z=x)=2\pi b_{m'+\tilde m'}(t=0,x,z=x,p,p_1=\tilde \sigma)
\]
\[
=(2\pi)^2 a_{m'}(0,x,z=x,p,\tilde \sigma=p_1) \tilde a_{\tilde m'}(t=0,z=x,z=x,p, \tilde \sigma=p_1) 
\]
\[
= \sigma_{0}(u)(x,z=x) \sigma_{0}(v)(z=x,x).
\]

The symbol in the flow-out is the sum of the principal symbols of each summand at ($t=z_1-x_1,x'=z',z_1-t+s=x_1+s$), which after the cancellation of $\chi_1,\chi_2$, reduces to
\[
\begin{aligned}
\sqrt{2\pi} \int e^{it\tilde \sigma} & e^{is (\sigma-\tilde \sigma)} a(s,x,(x_1+s,x'),(0,p'),\sigma) \\
& \tilde a(t-s,(x-1+s,x'),(z_1,x'),(0,p'),\tilde \sigma) ds d\sigma d\tilde \sigma
\end{aligned}
\]
\[
\begin{aligned}
=\frac{1}{\sqrt{2\pi}} \int & \left( \sqrt{2\pi} \int e^{is\sigma} a(s,x,(x_1+s,x'),(0,p'),\sigma)d\sigma \right) \\
& \left( \sqrt{2\pi} \int e^{i (t-s) \tilde \sigma} \tilde a(t-s,(x_1+s,x'),(z_1,x'),(0,p'),\tilde \sigma) d\tilde \sigma \right) ds
\end{aligned}
\]
\[
=\frac{1}{\sqrt{2\pi}}\int \sigma_{1}(u)(x,(x_1+s,x'),(0,p'))\sigma_{1}(v)((x_1+s,x'),(z_1,x'),(0,p'))ds.
\]
\end{proof}

%------section--------

\section{The manifold case}

\label{sec:GeneralCase}

\subsection{The spaces $J^{\ell,m}$ on manifolds}

\label{Sec:DefOnManifolds}

In this section we extend the definition of the spaces $J^{\ell,m}$ to the manifold case.
Let $M$ be a $C^{\infty}$ manifold of dimension $n$, and $X_c \subset T^{*}M$ be a compact domain
with smooth boundary contained in the cotangent bundle. The boundary $\partial X_c$ is then foliated  by 
curves tangent to the kernel of the pull-back of the symplectic form.  In addition, we assume that the 
fibrating and Bohr-Sommerfeld conditions are satisfied, i.e., the leaves of the null foliation are the fibers 
of a submersion $\pi: \partial X_c\to S$, and for each leaf 
$\gamma \subset \partial X_c$, 
\begin{equation}
\label{Eq:BohrSomm}
\int_\gamma \alpha \in 2\pi \mathbb{Z},
\end{equation}
where $\alpha$ is the tautological one-form in $T^* M$.

\begin{definition} Given $X_c$ as above, define the flow-out
\begin{equation}
\label{Def:FlowOutXc}
\mathcal{F}\partial X_c:=\left\{ (\x, \y) ~\big |~ \x, \y \in \partial X_c, ~\x, \y \text{ are in the same leaf} \right\} \subset T^*\left( M\times M\right).
\end{equation}
\end{definition}
The condition that $\partial X_c$ be fibrating easily implies that $\mathcal{F}\partial X$ is a closed
submanifold of $T^*\left( M\times M\right)$.

\begin{lemma}
\label{Lemma:LocalP}
There exists an open neighborhood $\partial X_c \subset U \subset T^* M$ of $\partial X_c$, and a map
\begin{equation}
\label{eq:LocalP}
P:U \to \mathbb{R}
\end{equation}
such that zero is a regular value of $P$, $\partial X_c=P^{-1}(0)$,  the orbits of the Hamilton flow generated by $P$ are $2\pi$ periodic on the boundary $\partial X_c$, and coincide with the leaves of the foliation.
\end{lemma}

\begin{remark}
From now on we will fix a defining function of $X_{c}$, $P$, with the properties of this Lemma.
\end{remark}

\begin{proof}
There exists $U$ and $F: U\to\bbR $ such that $\partial X_c=F^{-1}(0)$, where zero is a regular value. 
The Hamiltonian vector field $\Xi_F$ is tangent to $F^{-1}(0)=\partial X_c$,
and therefore the trajectories of the Hamiltonian $F$ coincide set-theoretically
with the leaves of the foliation.  In particular, they are periodic.  For each $\x\in \partial X_c$, let $T(\x)$ 
denote the minimal period of the trajectory through $\x$.  The fibrating condition can be seen to imply 
that the function $T$ is smooth.  Extend this function to a smooth function
$T: U\to \mathbb{R}^{+}$. The defining function that 
satisfies the conclusions of the lemma is then
\[
P=\frac{T}{2\pi} ~F
\]
\end{proof}

\begin{remark}
Notice that the boundary $\partial X_c$ of $X_c$ may not be connected. In those cases, the flow-out 
$\mathcal{F}\partial X_c$ consist of the union of Lagrangian that are pairwise disjoint.
\end{remark}

Let $\Delta \subset T^*\left( M\times M \right)$ be the diagonal in $T^*\left( M\times M\right)$. 
\begin{lemma}
The diagonal and the flow-out $(\Delta, \F)$ intersect cleanly.
\end{lemma}

\begin{proof}
Let $(\x,\x) \in \Delta \cap \F$. It is easy to show that
\[
T_{(\x,\x )}\F=\left\{ (\delta\x+r \Xi_{P}(\x),\delta\x) \big | \delta\x \in T_{\x}\partial X_c, r \in \mathbb{R} \right\},
\]
and so
\[
T_{(\x,\x)}\Delta \cap T_{(\x,\x )} \F =\left\{ ( \delta\x,\delta\x ) \big | \delta\x \in T_{\x} \partial X_c \right\} = T_{(\x,\x)} \left( \Delta \cap \F \right)
\]
\end{proof}

By an analogue to Proposition 2.1 in \cite{GU}, there exists a locally finite covering $\{U_i\}_i$
of $\Delta \cap \mathcal{F}\partial X_c$, $U_i\subset T^*(M\times M)$ open and contractible, where each $U_i$ intersects only one connected component of $\mathcal{F}\partial X_c$, and for each $U_i$ a canonical transformation
\[
\chi_i : U_i\longrightarrow T^{*}\mathbb{R}^{2n}
\]
mapping $U_i \cap \Delta$ to $\Delta^n$, and $U_i \cap \F$ to $\Fn$, where $\Delta^n, \Fn$ are the diagonal and flow-out in the model case, respectively.

\begin{definition}
\label{def:JGenCase}
Let $\Sigma=\Delta \cap \mathcal{F} \partial X_c$ be the intersection of the diagonal and the flow-out. The space $J^{\ell,m}\left( M\times M; \Delta, \mathcal{F}\partial X_c \right)$  is the set of families of functions $A(x,y,\h)$ which can be written in the form
\[
A=A_0+A_1+\sum \omega_i
\]
where:
\begin{enumerate} 
\item $A_0 \in {\text sc-}I^{\ell+m}\left( M\times M ; \Delta \setminus \Sigma \right)$ and $A_1 \in {\text sc-}I^{\ell}\left( M\times M;\F\right)$, and 
\item $\omega_i$ is microlocally supported in $U_i$ and is of the form 
\[
\omega_i = F_i (v_i),
\]
where $v_i \in  J^{\ell,m}\left( \mathbb{R}^{n} \times \mathbb{R}^n; \Delta^n, \Fn \right)$ and $F_i$ is a semiclassical Fourier integral operator associated to $\chi_{i}^{-1}$.
\end{enumerate}
\end{definition}

\begin{remark}
As in \cite{GU}, one can show that 
the definition does not depend on the choice of the semiclassical FIOs $F_i$.
\end{remark}

\begin{proposition}\label{FSetGenCase}
In the general case the conclusion of 
Proposition \ref{FSetModelCase} still holds, namely:  
Operators with kernel $A\in J^{\ell,m}\left( M\times M;~\Delta ,\F \right)$ have frequency set contained in the union
\[
FS(A)\subset \Delta' \cup \F '.
\]
Away from the intersection $\Sigma$ they are microlocally semiclassical Fourier integral operators ${\text sc-}I^{\ell+m}\left( M\times M ; \Delta \setminus \Sigma \right)$ and ${\text sc-}I^{\ell}\left( M\times M,\mathcal{F}\partial X_c \setminus \Sigma \right)$ respectively. Furthermore, one has well-defined symbol maps (ignoring Maslov factors)
\begin{equation}\label{symbolsGenCase}
\begin{array}{ccc}
 &\sigma_0\nearrow  & |\wedge|^{1/2}(\Delta\setminus\Sigma)\\
 J^{\ell,m}\left( M \times M;~\Delta,\mathcal{F}\partial X_c \right) & \\
 & \sigma_1\searrow &  |\wedge|^{1/2}(\mathcal{F}\partial X_c \setminus\Sigma).
\end{array}
\end{equation}
\end{proposition}
The proof of the existence of the symbol calculus is the subject of the following section.

\medskip
Our construction, in particular, gives a way to associating to $X_c$ an algebra of operators, which can be thought of 
as a quantization of $X_c$:
\begin{definition}
We will denote by $\mathcal{A}_{X_c}$ the space of  elements in $J\left( M\times M; \Delta, \partial X_c \right)$ that
are microlocally of order O($\h^\infty$) in the complement of $X_c$.
\end{definition}
It is not hard to see that $\mathcal{A}_{X_c}$ is indeed closed under composition.  (Elements in this algebra 
correspond to  amplitudes that are of order $-\infty$ in $\sigma$ as $\sigma\to -\infty$.)

%---subsection-----

\subsection{Symbolic calculus}

\label{sec:SymbCalculus}

Here we discuss how Proposition
\ref{FSetGenCase} can be proved (for $\h$ tending to zero along certain sequences) using the methods from \cite{PU}.  
We begin by recalling the ideas and results from {\it op.~ cit.} that we will need.  

The pre-quantum circle bundle of $T^*M$ can be identified with the following submanifold of $T^*(M\times S^1)$:
\[
Z = \{(x,\theta; \xi,\kappa)\in T^*(M\times S^1)\;;\; \kappa =1\},
\]
with the obvious circle action.  The connection form, $\alpha$, is the pull-back to $Z$ of the canonical one form of $T^*(M\times S^1)$.  

\begin{definition} (\cite{PU})
A Lagrangian submanifold $\Lambda\subset T^*M$ will be called {\em admissible} iff  there exists a conic Lagrangian submanifold, 
\[
\tilde\Lambda\subset T^*(M\times S^1)\cap \{\kappa >0\}
\]
such that 
\begin{equation}
\label{eq:Homogenization}
\Lambda = \Bigl(\tilde\Lambda\cap Z\Bigr)/S^1.
\end{equation}
We call such a $\tilde\Lambda$ a {\em homogenization} of $\Lambda$.
\end{definition}

It is not hard to see that $\Lambda$ is admissible if and only if it satisfies the following
Bohr-Sommerfeld condition:  There exists $\phi:\Lambda\to S^1$ such that
\[
\iota^*\eta = d\log\phi,
\]
where $\iota:\Lambda\to T^*M$ is the inclusion and $\eta$ the canonical one form of $T^*M$.  Given such a $\phi$, a homogenization of $\Lambda$ can be defined by:
\begin{equation}
\label{eq:hom_phi}
\tilde\Lambda = \{e^{i\theta}=\phi(\lambda), \ \kappa>0\}.
\end{equation}

\begin{definition}
Let $M$ be a $C^{\infty}$ manifold and consider a family of smooth functions $\left\{ \psi_{\hbar} \right\}$. The \emph{$\hbar$ transform} of the family  $\psi_{\hbar}$ is the following distribution (if the series converges weakly) in $M\times S^{1}$:
\[
\Psi(x,\theta)=\sum_{m=0}^{\infty}\psi_{1/m}(x)e^{im\theta} .
\]
\end{definition}

The main point of the previous two definitions is the following
\begin{lemma}(\cite{PU})
The $\h$ transform of a semiclassical state associated to an admissible Lagrangian is a Lagrangian distribution associated to a homogenization of the Lagrangian submanifold.
\end{lemma} 

We claim that the previous lemma generalizes to our spaces 
$J^{\ell, m}$ of semiclassical pseudodifferential operators with singular symbols, so that the
$\h$ transform of their kernels are Guillemin-Uhlmann operators (i.~e.~ those defined in \cite{GU}) on $M\times S^1$.  

Thus we now consider $\Delta$ to be the diagonal in $T^*\left( M\times M\right)$ and $\mathcal{F}\partial X_c$ as in \eqref{Def:FlowOutXc}.  Clearly a homogenization for $\Delta$ is the diagonal $\widetilde{\Delta}$ of $T^{*}\left(M\times S^{1}\right)^{+}=T^{*}\left( M\times S^1 \right) \cap \left\{ \kappa > 0\right\}$.  Let $P$ be a globally defining function for $\partial X_c$ with periodic Hamilton flow on it,
as in \eqref{eq:LocalP}.
Then a homogenization for $\mathcal{F}\partial X_c$ is the flow-out of the homogenization of $P$,  which is the function $\tilde{P}: T^{*}\left(U\times S^{1}\right)^{+}\to\bbR$ defined as
\[
\widetilde{P}(x, e^{i\theta}; \xi,\kappa) = \kappa P(x,\xi/\kappa),
\]
where $U$ is the neighborhood of $\partial X_c$ described in Lemma \ref{Lemma:LocalP}. Specifically, 
\begin{equation}
\label{eq:LiftedFlow}
\widetilde{\mathcal{F}\partial X_c}=\left\{ \left( \Phi_s^{\widetilde{P}} (x, e^{i\theta}, \xi,\kappa)~;~(x, e^{i\theta}, \xi,\kappa ) \right)  \big| ~s\in \mathbb{R}, \kappa > 0 ,~P(x,\xi/\kappa)=0 \right\}
\end{equation}
where $\Phi_s^{\widetilde{P}}$ is the Hamilton flow generated by the equations (here $p=\xi/\kappa$)
\[
\begin{array}{llll}
\dot{x}=\frac{\partial P}{\partial p}(x,p), & &\dot{\theta}=-p~\frac{\partial P}{\partial p}(x,p) \\ \\
\dot{\xi}=-\kappa \frac{\partial P}{\partial x}(x,p) ,& & \dot{\kappa}=0.
\end{array}
\]
Notice that $\widetilde{\mathcal{F}\partial X_c}$ is closed if the flow $\Phi_t^P$ is $2\pi$ periodic, which happens if the Bohr-Sommerfeld condition
\[
\int_{\gamma}pdx \in 2\pi \mathbb{Z}
\]
is satisfied for orbits $\gamma \subset \partial X_c$ of $P$. 
The function $\phi: \F \to S^1$ associated to this homogenization is 
\[
\phi(\x,\y) = e^{-i \int_{\x}^{\y} pdx },
\]
where $\int_{\x}^{\y} pdx$ is the action from $\x$ to $\y$, and the integral is taken on the curve in the leaf joining $\x$ to $\y$.

%----subsection----

\subsection{The existence of the symbolic calculus. }
\label{hTransform}

Notice that $\Delta$ and $\mathcal{F}\partial X_c$ intersect cleanly. As a result, their homogenizations $\widetilde{\Delta}$ and $\widetilde{\mathcal{F}\partial X_c}$ are conic and intersect cleanly too, forming an intersecting pair, in the sense of \cite{MU,GU}. 
The relationship between the semiclassical objects defined above and the operators described in \cite{GU} is as follows:

\begin{theorem}
\label{SymbCalculus}
Let $\Delta, \mathcal{F}\partial X_c$ be as above, and $\widetilde{\Delta}, \widetilde{\mathcal{F}\partial X_c}$ their homogenization, respectively. Then any operator $A \in J^{\ell,m}\left( M\times M ~;~\Delta, \mathcal{F}\partial X_c \right)$ if and only if its $\h$-Transform belongs to $I^{\ell,m}\left( M \times S^1\times M \times S^1 ~;~  \widetilde{\Delta}, \widetilde{\mathcal{F}\partial X_c} \right)$.
\end{theorem}

\begin{proof}
We sketch the ideas of the proof in the model case. Let $A(x,y,\h)$ be as in equation \eqref{OscInt}. Let us assume first that the amplitude $\frac{1}{(2\pi \h)^n}a$ does not depend on $\h$ so that
\[
A(x,y,\h)=\int e^{i\h^{-1}[ (x_1-y_1-s)p_1-(x'-y')p' ]+is\sigma}a(s,x,y,p,\sigma)dsdp d\sigma ,
\]
where $a(s,x,y,p,\sigma)$ is a classical symbol in $\sigma$ of order $m'=m-1/2$, and compactly supported in $s,x,y,p$. 
The $\h$-transform of $A$ is then
\[
A(x,\theta,y,\alpha)=\sum_{k=1}^{\infty} \int e^{ik [(x_1-y_1-s) p_1+(x'-y') p'+(\theta-\alpha)]+i s\sigma}a(s,x,y,p,\sigma)ds dp d\sigma
\]
\[
=\int \frac{e^{i[ (x_1-y_1-s)p_1+(x'-y')p'+(\theta-\alpha) ]}}{1-e^{i[ (x_1-y_1-s)p_1+(x'-y') p'+(\theta-\alpha) ]}}e^{is\sigma}a(s,x,y,p,\sigma)dsdpd\sigma .
\]
This distribution is the push-forward under the projection $(s,x,y,\theta,\alpha,p)\to (x,y,\theta,\alpha)$ of the product of the
distributions
\[
\Gamma(s,x,y,\theta,\alpha,p) = \frac{e^{i[ (x_1-y_1-s)p_1+(x'-y') p'+(\theta-\alpha) ]}}{1-e^{i[ (x_1-y_1-s)p_1+(x'-y') p'+(\theta-\alpha) ]}}
\]
and
\[
\Upsilon(s,x,y,\theta,\alpha,p) = \int e^{is\sigma}a(s,x,y,p,\sigma) d\sigma .
\]
$\Gamma$ 
is a distribution in space conormal to the hypersurface 
$(x_1-y_1-s)p_1+(x'-y') p'=-(\theta-\alpha)$,
while $\Upsilon$ is a distribution conormal to $s=0$.
Therefore, $A(x,\theta,y,\alpha)$ is a Guillemin-Uhlmann distribution associated to the pair $\left( \widetilde{\Delta^n} , \widetilde{\Fn} \right)$. 
The general case (i.e. when $a$ also depends on $\h$) is an asymptotic sum of derivatives and integrals (with respect to the $\theta$ variables) of the 
previous case.  The converse is also true, and the proof is analogue to that in \cite{PU} for semiclassical states.
\end{proof}

This proposition together with Proposition 2.4 of \cite{PU}, relating the frequency set of an $\h$-dependent vector and
the wave-front set of its $\h$ transform, implies part (1) of Theorem \ref{Main}.

The symbols of operators in $J^{\ell, m}(\Delta, \mathcal{F}\partial X_c)$ are the reduction of the symbols of its $\h$-Transform,
in the following sense. Let $\x,\y \in \mathcal{F}\partial X_c \setminus \Sigma$, $\x=\Phi_s^{P}(\y)$ for some $s\in \mathbb{R}$, and take 
\[
(\widetilde x,\widetilde y) =\left( \Phi_s^{\widetilde{P}} (\y, e^{i\theta=0}=1,\kappa=1), (\y, e^{i\theta=0}=1,\kappa=1)\right) \in \widetilde \Fn.
\]
Using \eqref{eq:LiftedFlow}, we obtain an isomorphism between $T_{(\widetilde x,\widetilde y)} \widetilde{\mathcal{F}\partial X_c}$ and $T_{(\x,\y)} \mathcal{F}\partial X_c \times \mathbb{R} \times \mathbb{R}$, which leads to
\begin{equation}
\label{eq:SymbolDownstairs}
\left | T_{(\widetilde x, \widetilde y)} \widetilde{\mathcal{F}\partial X_c} \right |^{1/2} \cong \left | T_{(\x, \y)} \mathcal{F}\partial X_c \times \mathbb{R}\times \mathbb{R} \right|^{1/2} \cong \left | T_{(\x,\y)} \mathcal{F}\partial X_c \right |^{1/2}
\end{equation}

Therefore, every half-density in $T_{(\widetilde x, \widetilde y) } \widetilde{\mathcal{F}\partial X_c}$ will define a half-density in $T_{(\x,\y)} \mathcal{F}\partial X_c$. Let $\widetilde{\Sigma}=\widetilde{\mathcal{F}\partial X_c} \cap \widetilde{\Delta}$. For each family $A\in J^{\ell,m}$, denote its $\h$-Transform by $\widetilde A$. There is a well defined symbol map 
\[
\widetilde{\sigma_1}=\sigma_1\left( \widetilde A \right)_{\widetilde{\mathcal{F}\partial X_c}-\widetilde{\Sigma}} \in 
C^{\infty}\left( \widetilde{\mathcal{F}\partial X_c}\setminus\widetilde{\Sigma},\widetilde{\Omega_1}\bigotimes \widetilde{L_1} \right),
\]
where $\widetilde{\Omega}_{j}$ is the bundle of half-densities on $\widetilde{\mathcal{F}\partial X_c}$, and $\widetilde{L_i}$ is the corresponding Maslov bundle. Ignoring Maslov factors, we can define the symbols on $\mathcal{F}\partial X_c \setminus \Sigma$ by the identification \eqref{eq:SymbolDownstairs}, i.e., restricting the symbol of the $\h$-Transform to $\theta=0$, $\kappa=1$ ($\sigma_1=\widetilde{\sigma_1}_{\big|_{\theta=0,\kappa=1}}$). The construction of the principal symbol on the diagonal is similar.

\bigskip
Let $\Sigma^n=\Delta^n \cap \Fn$, $\widetilde{\Sigma}^n=\widetilde{\Delta^n} \cap \widetilde{\Fn}$.  In the model case, we know by \cite{GU} that $A(x,\theta,y,\alpha)$ is microlocally in 
\[
I^{\ell+m}\left( \mathbb{R}^{n}\times S^1 \times \mathbb{R}^n \times S^1 ~;~ \widetilde{\Delta^n}\setminus \widetilde{\Sigma}^n \right) , \text{ and }I^{\ell}\left( \mathbb{R}^{n} \times S^1 \times \mathbb{R}^{n} \times S^1 ~;~\widetilde{\Fn} \setminus \widetilde{\Sigma}^n \right).
\]
Therefore $A_{\hbar}$ is microlocally in 
$\text{sc-}I^{\ell+m}\left( \bbR^n \times \bbR^n ; \Delta^n \setminus \Sigma^n \right)$ and \\
 $\text{sc-}I^{\ell} \left( \bbR^{n} \times \bbR^n ; \Fn 
\setminus \Sigma^n \right)$ in the sense that the $\h$-transform is microlocally in their corresponding spaces. We have proved the following

\begin{proposition}
\label{IntersJs}
\[
\cap_{\ell}J^{\ell,m}= \text{sc-} I^{\infty}\left( \bbR^n \times \bbR^n ; \Delta^{n} \right)
\]
and
\[
\cap_{m}J^{\ell,m}=\text{sc-} I^{\ell}\left( \bbR^n \times \bbR^n ; \Fn \right).
\]
\end{proposition} 

%-------subsection--------

\subsection{A symbolic compatibility condition}

Recall that the foliation of $\partial X_c$ is fibrating, i. e., there exists a $C^\infty$ Hausdorff manifold $S$ and a smooth fiber map 
\begin{equation}
\label{eq:fibration}
 \pi:\partial X_c \to S,
\end{equation}
whose fibers are the connected leaves of the foliation defined Section \ref{Sec:DefOnManifolds}. Elements of $\mathcal{F}\partial X_c$ are pairs 
of points in $\partial X_c$ that lie in the same fiber of $ \pi$. 

Generalizing a construction in \cite{GuS2},  given $s\in S$ 
let $\mathfrak{SO}_s$ be the $*-$algebra of all pseudodifferential operators, acting on the space of half-densities 
$C^{\infty}\left( |F_s|^{1/2} \right)$, where $F_s$ is the fiber of 
$\pi: \partial X_c \to S$ above $s$. 
Let $\mathfrak{SO}$ be the sheaf of $*-$algebras on $S$ whose stalk at $s$ is $\mathfrak{SO}_s$. We will say that a section $k$ of 
$\mathfrak{SO}$ is smooth if the Schwartz kernel of the operator $k(s)$ depends smoothly on $s \in S$. 
The Schwartz kernel theorem, applied fiber-wise to the fibers of $\pi$, together with the natural symplectic structure
of $S$ yield the following:

\begin{proposition}
\label{prop:AlgerbaSheaf}
The vector space of smooth sections of the sheaf $\mathfrak{SO}$ is naturally 
isomorphic to the space of half-densities on $\mathcal{F} \partial X_c \setminus \Sigma$ that extend to $\mathcal{F} \partial X_c$ as a conormal distribution to $\Sigma$
\end{proposition}
\begin{proof}
Let  $\gamma =( \gamma_1, \gamma_2) \in  \F$, $s= \pi( \gamma_1)= \pi( \gamma_2) \in S$. 
One then has the following fiber product diagram:
\[
\xymatrix {
& T_{\gamma} \F  \ar@{->}^{d\pi_1}[r] \ar@{->}_{d\pi_2}[d] & T_{ \gamma_1}\partial X_c \ar@{->}^{d  \pi}[d] \\
& T_{ \gamma_2}\partial X_c \ar@{->}^{d \pi}[r] & T_{s}S
}
\]
and so we get the exact sequence
\[
\begin{array}{lll}
\xymatrix{
0 \ar@{->}[r] & T_{\gamma} \F \ar@{->}[r] & T_{ \gamma_1}\partial X_c \oplus T_{ \gamma_2} \partial X_c \ar@{->}[r] & T_{s}S  \ar@{->}[r] & 0
} \\
\xymatrix{
&~~~~~~~~ v~~~~~~ \ar@{|->}[r] &~~~ \left( d\pi_1 v, d\pi_2 v \right) & &
} \\
\xymatrix{
& & ~~~~~~~~~~~~~~~~~~~~~~(v_1, v_2) ~~~~~~~~\ar@{|->}[r] &~~~ d \pi (v_2)-d \pi (v_1) &
}
\end{array}
\]
This gives a natural identification
\[
\mathbb{C}  \cong  \left| T_{\gamma} \F \right|^{1/2}  \otimes  
\left| T_{ \gamma_1}\partial X_c \oplus T_{ \gamma_2}\partial X_c \right|^{-1/2} \otimes  \left| T_s S\right|^{1/2} .
\]
Since $S$ is a symplectic manifold, there is a canonical half-density on $T_s S$, and we get an identification
\[
\big| T_{\gamma}  \F \big|^{1/2} \cong \big| T_{ \gamma_1} \partial X_c \oplus T_{ \gamma_2} \partial X_c \big|^{1/2} \cong \big| T_{ \gamma_1}\partial X_c \big|^{1/2} \otimes \big| T_{ \gamma_2} \partial X_c \big|^{1/2}.
\]
Finally, given a half density in $T_{ \gamma_k}\partial X_c$, $k=1,2$, we need to get a half-density in $T_{\gamma_k}F_s$, where $F_s$ is the fiber of $\pi$ above $s$. We have the following exact sequence:
\[
\xymatrix{
0 \ar@{->}[r] & \ker~d \pi_{ \gamma_k} \ar@{->}[r] & T_{ \gamma_k}\partial X_c \ar@{->}[r] & T_s S \ar@{->}[r] & 0
}
\]
which, by the same process as above, gives an identification
\[
\big| T_{\gamma_k}F_s \big|^{1/2} = \big| \ker~d \pi_{ \gamma_k} \big|^{1/2} \cong \big| T_{\gamma_k} \partial X_c \big|^{1/2}
\]
Hence
\[
\big| T_{\gamma}  \F \big|^{1/2} \cong \big| T_{\gamma_1}  F_s \big|^{1/2}\otimes \big| T_{\gamma_2}  F_s \big|^{1/2}
\cong \big| T_{(\gamma_1,\gamma_2)} (F_s\times F_s) \big|^{1/2}
\]
This gives a smooth isomorphism between two line bundles over $\F$:  $\big| T \F \big|^{1/2}$ and the line bundle 
$\Upsilon\to\F$ whose fiber over $(\gamma_1, \gamma_2)$ is $ \big| T_{(\gamma_1,\gamma_2)} (F_s\times F_s) \big|^{1/2}$,
where $\pi(\gamma_1) = s = \pi(\gamma_2)$.  Clearly a section of the sheaf $\mathfrak{SO}$ is a distributional section
of $\Upsilon$ conormal to $\Sigma$, and by the previous isomorphism this is equivalent to a distributional 
section of $\big| T \F \big|^{1/2}$ conormal to $\Sigma$.
\end{proof}

The previous isomorphism yields an algebra structure on the space of distributional half densities on $\mathcal F\partial X_c$
which are conormal to $\Sigma$.  Analogously as in \cite{GuS2} (Proposition 2.7), one can see that the algebraic structure
on this space is given, away from $\Sigma$, by the composition of half densities regarded as symbols of Fourier integral operators
associated to $\mathcal F\partial X_c$.  

\medskip
Let us now take $A\in J^{\ell,m}$.  The symbol $\sigma_1 (A)$ in $\mathcal{F}\partial X_c \setminus \Sigma$ blows-up as the point 
where $\sigma_1$ is evaluated tends to the intersection. In fact, in \cite{AUh} it was shown that this symbol
has a natural extension to a 
distribution conormal to $\Sigma$. 
Using the same identification above, this determines the kernel of a Pseudodifferential operator on the fiber 
above each point of $S$.  We have proved:

\begin{proposition}
\label{prop:Symb1PsiDO}
For $A\in J^{\ell,m}\left( M\times M, \Delta, \mathcal{F}\partial X_c \right)$, the symbol $\sigma_1(A)$ can be identified with 
a global section of the sheaf $\mathfrak{SO}$, that is, with a family of
classical pseudodifferential operators of order $m'=m-1/2$ acting on fibers of $\pi: \partial X_c\to S$ (orbits in the flow-out).
\end{proposition}

For each $s\in S$ and $F_s$ the corresponding fiber above $s$, let us denote this operator by ${\sigma_1(A)}_s$:
\[
\xymatrix{
C^{\infty}\left( F_s \right) \ar@{->}^{{\sigma_1(A)}_s}[r] & C^{\infty}\left( F_s \right)
}
\]

As a result, there is a well-defined symbol
\[
\sigma\left( \sigma(A)_s \right) :T^{*}F_s\setminus 0 \to \mathbb{C}.
\]
As we will now see, this symbol is related to the symbol $\sigma_0(A)$ of $A$ on the diagonal.  
(This is the compatibility of the symbols of $A$ announced earlier.)

Let $P$ be a defining function of $\partial X_c$ with a $2\pi$-periodic flow on $\partial X_c$.
We have the following diffeomorphism
\[
\begin{array}{cc}
\xymatrix{
\partial X_c \times S^1 \ar@{->}[r] & \mathcal{F}\partial X_c
} \\
\xymatrix{
(\x,t) ~\ar@{|->}[r] & ~~ \left(\Phi_{t}^{P}\x , \x \right)
}
\end{array}
\]
from which one can see that, for any $\gamma =(\x,\x) \in \Sigma =\Delta \cap \mathcal{F}\partial X_c$, 
there is a natural isomorphism
\[
N_{\Sigma}^{\F}:= T_{\gamma}\F / T_{\gamma}\Sigma \cong T_{\x}F_s 
\] 
($N_{\Sigma}^{\F}$ is the normal space to $\Sigma$ in $\F$ at $\gamma$). 
Therefore, for each $\x \in \partial X_c$, $T_{\x}^{*}F_s $ is isomorphic to the dual space 
$\left(N_{\Sigma}^{\F}\right)^*$.

Now let $N_\Sigma^{\Delta}:=T_{\gamma} \Delta / T_{\gamma} \Sigma$ be the normal space to 
$\Sigma$ in $\Delta$ at $\gamma$.  By \cite{GU}, $N_\Sigma^{\mathcal{F}\partial X_c}$ and $N_\Sigma^{\Delta}$ are 
supplementary Lagrangian subspaces of the two-dimensional symplectic vector space $W=\left( T_{\gamma} \Sigma \right)^{\perp} / T_{\gamma} \Sigma$.
Therefore, $N_\Sigma^{\mathcal{F}\partial X_c}$ and $N_\Sigma^{\Delta}$ are in duality with each other
(they are canonically paired by the symplectic form of $W$).  
In the end we obtain a natural isomorphism
\[
T_{\x}^{*}F_s  \cong \left(N_{\Sigma}^{\F}\right)^* \cong N_\Sigma^{\Delta}.
\]

The symbol of $A$ on the diagonal belongs to a class $S^{m'}\left( \Omega_0; \Delta, \Sigma \right)$ of smooth functions
on $\Delta\setminus\Sigma$ which blow up at a prescribed rate at $\Sigma$ (see \cite{GU} for more details). 
Every element in this class determines a smooth function on $N_\Sigma^{\Delta}\setminus\{0\}$. 
The \emph{compatibility condition} alluded to in the Introduction is the following:

\begin{theorem}
\label{Thm:SymbComp}
The symbol of $A\in J^{\ell,m}\left( M\times M, \Delta, \mathcal{F}\partial X_c \right)$
on the flow-out, identified with a family $\{\sigma_1(A)_s\}_{s\in S}$
of pseudodifferential operators along the fibers $F_s$, 
satisfies that for each $\x \in F_s$
\begin{equation}
\label{SymbCompCond}
\begin{array}{ll}
\sigma(\sigma_1(A)_s)_{\x} (\tau) = \lim_{\substack{\y \to \x \\ \y \in X_c\setminus \partial X_c }} \frac{\sigma_{0}(A)(\y,\y)}{P^{m'}(\y)}
 \\ \text{and}\\  \\
\sigma(\sigma_1(A)_s)_{\x} (-\tau) = \lim_{\substack{\y \to \x \\ \y \in T^{*}M \setminus X_c}} \frac{\sigma_{0}(A)(\y,\y)}{P^{m'}(\y)},
\end{array}
\end{equation}
where $\tau\in T_{\x}^*F_s$ is the dual of the Hamilton field of $P$ at $\x$ regarded as an element in $T_{\x}F_s$,
and the two limits are taken from the interior of $X_c$ and the exterior of $X_c$, respectively. Moreover, this condition is intrinsic, i.e., it does not depend on the choice of $P$.
\end{theorem}

The proof reduces to the model case, where it is immediate.  One can also verify directly that changing $P$ by a multiplicative factor
does not alter the relationships (\ref{SymbCompCond}).

The symbols in the flow-out become more natural under the present setting, as can be seen in the following symbolic version of Proposition \ref{CompoModelCase}.
\begin{proposition}\label{prop:SymbCompoModelCase}
The composition of properly supported operators with kernels in $J^{\ell, m}$ and $J^{\tilde\ell, \tilde m}$, respectively, is an operator with kernel in $J^{\ell+\tilde\ell+1/2, m+\tilde m-1/2}$. For any $A \in J^{\ell,m}$, $B \in J^{\tilde \ell,\tilde m}$, we obtain the usual symbol in the diagonal:
\[
\forall \x\in X_c\setminus\partial X_c\quad
\sigma_0 (A \circ B)(\bar x, \bar x)=\sigma_0 (A) (\bar x,\bar x) \sigma_0(B) (\bar x, \bar x),
\]
and for any fiber $F_s$ above $s\in S$,
\[
{\sigma_{1}}(A\circ B)_s=\sigma_{1}(A)_s \circ \sigma_{1}(B)_s\ .
\]
\end{proposition}

%----subsection-----

\subsection{The adjoint}

The class $J^{\ell,m}$ is closed under the operation of taking adjoints, and information about the symbol of the adjoint 
is given in the following
\begin{proposition}
Let $A\in J^{\ell,m}(M\times M; \Delta, \F)$, then the adjoint $A^*$ belongs again to $J^{\ell,m}(M\times M; \Delta, \F)$, and 
\[
\begin{array}{lll}
\sigma_0(A^*)(\x, \x )=\overline{\sigma_0(A)(\x,\x)}\quad \text{ for }   (\x,\x) \in \Delta \setminus \Sigma,\ \text{and} \\ \\
\sigma_1(A^*)_s=\left( \sigma_1(A)_s \right)^*,\quad \text{for each  $s\in S$.}
\end{array}
\]
\end{proposition}

\begin{proof}
The first statement is as in the usual theory of $\h$-pseudodifferential operators. 
It is enough to prove the rest in the model case. Take $A$ with Schwartz kernel
\[
K(x,y,h)=\frac{1}{(2\pi \h)^n} \int e^{i\h^{-1} ((x-y) \cdot p-s p_1)+is\sigma}a(s,x,y,p,\sigma,\h )ds dp d\sigma,
\]
where $a\in \mathcal{A}_{classical}^{\ell',m'}$, $\ell'=\ell+1/2$, $m'=m-1/2$. The Schwartz kernel of the adjoint is given by
\[
K^* (x,y)=\overline{K(y,x)}=\frac{1}{(2\pi \h)^n} \int e^{i\h^{-1} ((x-y)p-sp_1)+is\sigma}\overline{a(-s,y,x,p,\sigma,\h )} ds dp d\sigma,
\]
where $a(s,x,p,\sigma,\h)$ was replaced by $\overline{a(-s,y,x,p,\sigma,\h )}$.  
Taking $\x=(x,(p_1=0,p')), \y =\phi_s^P \x=((x_1+s,x'),(p_1=0,p'))$, the symbol in the flow-out is given by
\[
\sigma_1(A^*)(\phi_s^P \x,\x)=\sqrt{2\pi}\int \overline{a_{-\ell'}(-s,y,x,p,\sigma)}e^{is\sigma} d\sigma \sqrt{dxdy_1dp'}_{\big |_{x'=y',p_1=0,s=x_1-y_1}}
\]
\[ 
=\sqrt{2\pi} \overline{\int a_{-\ell'}(-s,y,x,p,\sigma) e^{-is\sigma} d \sigma} \sqrt{ dx dy_1dp'}_{\big |_{x'=y',p_1=0,y_1-x_1=-2}}=\overline{\sigma_1(A)(\x,\phi_s^P \x)}
\]
The proof is now clear, since the symbol $\sigma_1(A^*)$ intertwines the variables and takes the conjugate of $\sigma_1(A)$.
\end{proof}

%------section-------

\section{Asymptotics of the trace}

\label{sec:AsympTrace}

\subsection{The trace in case $m'\geq 0$.}
We now assume that $X_c$ is compact and $\partial X_c$ is fibrating, as in section \ref{sec:GeneralCase}. We will prove:

\begin{theorem}\label{Trace1}
Let $\widehat A$ be an operator in the class ${\mathcal A}_{X_c}$.  Then, if $m > 1/2$, 
\begin{equation}\label{}
\tr(\widehat A) = (2\pi)^{-n} \hbar^{-\ell-m-n} \int_{X_c}\sigma_0(x,x,p,-p)dxdp +O(\hbar^{-\ell-m-n+1}),
\end{equation}
where $\sigma_0$ is the symbol of $A$ on the diagonal. If $m=1/2$, 
\begin{equation}\label{}
\tr(\widehat A)=(2\pi)^{-n} \hbar^{-\ell-m-n} \int_{X_c}\sigma_0(x,x,p,-p)dxdp +O(\hbar^{-\ell-m-n} ~ \hbar ~\log(1/\hbar)).
\end{equation}
\end{theorem}

%----the proof----

The rest of this subsection is devoted to a proof of this result, which we break into a series of lemmas.
Note that it is enough to estimate the integral along the diagonal of the kernel of $A$ in the model case with $\ell'=0$. Consider therefore a semiclassical kernel of the form:
\begin{equation}\label{resfrio}
u_{\hbar}(x,y)=\frac{1}{(2\pi \hbar)^n} \int e^{i\hbar^{-1} [ (x_{1}-y_{1}-s)p_1+(x'-y') p' ] +is\sigma } a(s,x,y,p,\sigma)ds dp d\sigma,
\end{equation}
where $p=(p_1, p')$, $x=(x_1,x')$, $y=(y_1,y')$ and $a$ is a classical symbol in $\sigma$ of degree $m'=m-1/2$, compactly supported in $(s,x,y,p)$.

\begin{lemma}\label{LemmaRho}
Let $\mu_0 >0$ be large enough so that $a|_{\{|p_1|> \mu_{0}/2\}} = 0$, and let $\rho \in C_{0}^{\infty}$ be a smooth function with compact support which is equal to one in $[-\mu_{0}/2, \mu_{0}/2 ]$ and is supported in $[ -\mu_{0} , \mu_{0} ]$.  Then
\begin{equation}\label{Rho}
u_{\h}(x,x)=\frac{1}{(2\pi \h)^n} \int e^{-isp_{1}/\h +is\sigma}a(s,x,x,p,\sigma )\rho(\h \sigma)ds dp d\sigma + O(\h^\infty),
\end{equation}
uniformly in $x$.
\end{lemma}
\begin{proof}
The rigorous definition of $u_h$ when the integral in $\sigma$ diverges is given in equation \eqref{sense}:  If $K>>0$ ($K \ge m'/2+1$), then $u_h$ is equal to the absolutely convergent integral
\[
\frac{1}{(2\pi \h)^n}\int e^{i\h^{-1} \left[ (x_1-y_1-s)p_1+(x'-y')p' \right]+is\sigma} \frac{1}{\left( 1+\sigma^2 \right)^K}D^{K}\left[ e^{-is\h^{-1}p_1} a(s,x,y,p,\sigma) \right]ds dp d\sigma,
\]
and therefore
\[
u_{\h}(x,y)=\frac{1}{(2\pi \h)^n}
\sum_{j=0}^{K} \int e^{i\h^{-1} \left[ (x-y)p-sp_1 \right]+is\sigma}\left( -i\hbar^{-1}p_1 \right)^{2K-j}\frac{a_{j}(s,x,y,p,\sigma)}{\left( 1+\sigma^2\right)^K}ds dp d\sigma ,
\]
where the last expression is obtained by expanding the action of $D^K$.  Note that, $\forall j$, $a_j$ consists of linear combinations of derivatives of $a$ with respect to $s$ (and therefore  $a_j \in \mathcal{A}^{0,m'}$). Using this we get that the remainder in equation (\ref{Rho}) is equal to
\[
\sum_{j=0}^{K}\frac{1}{(2\pi \h)^n}\int e^{-isp_1/\h +is\sigma}\frac{a_{j}(s,x,x,p,\sigma)}{\left( 1+\sigma^2 \right)^{K}}\left( -i\h^{-1}p_1 \right)^{2K-j} \left( 1-\rho(\h \sigma) \right)ds dp d\sigma .
\]
Let $b_{j}(s,x,p,\sigma~;~\h)=\frac{a_{j}(s,x,x,p,\sigma)}{\left( 1+\sigma^2 \right)^K}(-i \h^{-1} p_1)^{2K-j} \left(1-\rho(\h \sigma) \right)$. We will show that for each $j$, 
\[
B_j:= \frac{1}{(2\pi \h)^n}\int e^{-isp_{1} \h^{-1}+is \sigma}b_{j}\, ds dp d\sigma  
\]
is ${O}(\h^{\infty})$

Starting with the change of variables $\mu =\h \sigma$, $\omega =-p_1+\mu$, we obtain that
\[
B_j  = \frac{\h^{-1}}{(2\pi \h)^n} \int e^{is \omega/\h}\, b_{j}(s,x,(-\omega+\mu,p'),\h^{-1} \mu~;\h)\, ds d\omega dp' d\mu
\]
\[
=\frac{1}{(2\pi \h)^n}\int e^{-i\h \xi_1 \xi_2}\,c_{j}(x,\mu,p'~;~\h~;~\xi)\,d\xi dp' d\mu,
\]
where $\xi=(\xi_1,\xi_2)$ are the dual variables to $(s,\omega)$, and
\[
c_{j}(x,\mu,p'~;~\h~;~\xi)=\frac{1}{2\pi}\int e^{-i (s,\omega)\cdot \xi}\,b_{j}(s,x,(-\omega+\mu,p'),\h^{-1}\mu~;~\h )ds d\omega
\]
is the Fourier transform of $b_j$ in the $(s,\omega)$ variables. Using the  inequality 
\begin{equation}\label{also-well-known}
\left| e^{it}-\sum_{k=0}^{N-1}\frac{(i t)^k}{k!} \right| \le \frac{|t|^N}{N!}, 
\end{equation}
we obtain that for each $N>0$,
\[
\int e^{-i\h \xi_1 \xi_2} c_{j}(x,\mu,p'~;~\h~;~\xi)d\xi-\sum_{k=0}^{N-1}\int \frac{\left( -i\h \xi_1 \xi_2 \right)^k}{k!} c_{j}(x,\mu,p'~;~\h~;~\xi)d\xi
\]
\[
=\int R_{N}(\xi ~;~\h)c_{j}d\xi,
\]
where $|R_{N}(\xi~;~\h)| \le \frac{\h^{N}|\xi_1 \xi_2|^N}{N!}$. Moreover, for each $k=0,\ldots,N-1$,
\[
\frac{(i\h)^k}{k!}\int (-\xi_1 \xi_2)^{k} c_{j} d\xi
=\frac{2\pi (-i\h)^k}{k!} \left( \frac{\partial^2}{\partial s \partial p_1} \right)^{k} b_{j}(s,x,p,\h^{-1} \mu~;~\h)_{\big|_{s=0,p_1=\mu}}
\]
\[
= \frac{2\pi (-i\h)^k\,\left( 1-\rho(\mu) \right)}{(1+(\h^{-1} \mu)^2)k!} \left( \frac{\partial^2}{\partial s \partial p_1} \right)^{k}a_{j}(s,x,(\mu,p'),\h^{-1} \mu)_{\big|_{s=0,p_1={\mu}}} =0, 
\]
since $\rho (\mu)$ is equal to one in the support of $a$. It follows that
\[
\left| \int e^{-i\h \xi_1 \xi_2}c_{j}d\xi \right| \le \frac{\h^N}{N!}\int \left| \left( \xi_1 \xi_2 \right)^N c_{j} \right|d\xi
\]
\[
= \frac{\h^N}{2\pi N!} \left| \left| \int e^{-i(s,\omega)\cdot \xi} \left( \frac{\partial^2}{\partial s \partial \omega} \right)^{N} \left( b_{j}(s,x,(-\omega+\mu,p'),\h^{-1}\mu~;~\h )\right) ds d\omega \right| \right|_{L_{\xi}^1},
\]
where $\norm{\cdot}_{L^1_\xi}$ denotes the $L^1$-norm of a function of the $\xi$ variables. The well-known inequality 
\begin{equation}\label{well-known-ineq}
\left|| \hat{v} \right||_{L^{1}(\mathbb{R}^d)} \le \sum_{|\alpha|\le d+1}\left||\partial^{\alpha}v\right||_{L^{1}(\mathbb{R}^{d})}
\end{equation}
implies that the above bound is in turn bounded by
\[
\frac{\h^N}{2\pi N!} \sum_{|\alpha| \le 3} \left| \left| \partial^{\alpha} \left( \frac{\partial^2}{\partial s \partial \omega} \right)^N\left[ b_{j}(s,x,(-\omega+\mu,p'),\h^{-1} \mu~;~\h) \right] \right| \right|_{L_{s,\omega}^{1}}
\]
\[
\begin{aligned}
=\frac{\h^N}{2\pi N!} \sum_{|\alpha| \le 3} \Big| \Big| & \frac{(1-\rho(\mu)) (-i\h^{-1})^{2K-j}}{ \left( 1+(\h^{-1} \mu)^2 \right)^K} \partial^{\alpha} \left( \frac{\partial^2}{ \partial s\partial \omega}\right)^{N} \\
&\left[ a_{j}(s,x,x,(-\omega+\mu,p'),\h^{-1} \mu ) (-\omega+\mu)^{2K-j} \right] \Big| \Big|_{L_{s,\omega}^1},
\end{aligned}
\]
where $\norm{\cdot}_{L_{s,\omega}^1}$ is defined similarly, and $\partial^{\alpha}=\frac{\partial^{\alpha_1}}{\partial s^{\alpha_1}}~\frac{\partial^{\alpha_2}}{\partial \omega^{\alpha_2}}$, $\alpha=(\alpha_1, \alpha_2)$. Therefore $B_j$ is bounded above by
\[
\begin{aligned}
\frac{\h^{-n+N-2 K+j}}{(2\pi)^{n+1}N!} \sum_{|\alpha|\le 3} \int \Big| \partial^{\alpha} \left( \frac{\partial^2}{\partial s \partial \omega} \right)^{N} & \Big[ \frac{(-\omega+\mu)^{2K-j} a_{j}(s,x,(-\omega+\mu,p'),\h^{-1}\mu)}{\left( 1+(\h^{-1} \mu)^2\right)^k}  \\
&(1-\rho(\mu))\Big] \Big| ds d\omega dp' d\mu
\end{aligned}
\]
\[
= \frac{\h^{-n+N+1-2K+j}}{(2\pi)^{n+1}N!} \sum_{|\alpha|\le 3} \int \left| \partial^{\alpha} \left( \frac{\partial^2}{\partial s \partial p_1} \right)^{N}\left[ \frac{p_1^{2K-j} a_{j}(s,x,p,\sigma )}{\left( 1+\sigma^2\right)^k} (1-\rho(\h \sigma))\right] \right| ds dp d\sigma.
\]
Finally, notice that the integrand is $O\left( \sigma^{m'-2K} \right)$ and that $1-\rho(\h \sigma)$ has support in $|\sigma | \ge \h^{-1} \mu_0/2$. Therefore the above upper bound is less than a constant times
\begin{alignat*}{1}
\h^{-n+N+1-2K+j}\int_{\h^{-1}\mu_0/2}^{\infty}\sigma^{m'-2K}d\sigma=& \h^{-n+N+1-2K+j}\frac{\sigma^{m'-2K+1}}{m'-2K+1}\left|_{\h^{-1}\mu_0/2}^{\infty} \right.
\\ &=O\left( \h^{j-n-m'+N} \right) .
\end{alignat*}
Since this is true for all positive integers $N$ we are done.
\end{proof}

\begin{lemma}
\label{lemmaTwo}
If $m'\geq 0$,
\begin{equation}\label{lemaDos}
\int u_\h(x,x)\, dx = \frac{1}{(2\pi \h)^n}\int 2\pi ~a(0,x,x,\mu, p', \h^{-1}\mu)\, dx\,dp'\,d\mu + O(\h^{-n-m'+1}).
\end{equation}
\end{lemma}
\begin{proof}
By (\ref{Rho}), it suffices to estimate
\[
\tilde{u}_\h(x,x) =
\frac{1}{(2\pi \h)^n} \int e^{-isp_{1}/\h +is\sigma}a(s,x,x,p,\sigma )\rho(\h \sigma)ds dp d\sigma.
\]
By an argument identical to the one used in the proof of the Lemma,
\[
\tilde{u}_\h(x,x)= \frac{1}{(2\pi \hbar)^n} ~\hbar^{-1} \int e^{i s \omega /\hbar} a(s,x,x,(-\omega+\mu,p'),\hbar^{-1} \mu) \rho(\mu) ds dp' d\omega d\mu.
\]
We will apply the method of stationary phase to the $(s,\omega)$ integral, before integrating with respect to $p'$ and $\mu$.   To this end introduce the notation
\[
u_\h(x,x,p',\mu) = \frac{1}{(2\pi \hbar)^n} ~\hbar^{-1} \int e^{i s \omega /\hbar} a(s,x,x,(-\omega+\mu,p'),\hbar^{-1} \mu) \rho(\mu) ds d \omega,
\]
so that the left-hand side of (\ref{lemaDos}) is equal to $\int u_\h(x,x,p',\mu)\, dx dp' d\mu$ modulo $O(\h^\infty)$.  We also have that
\[
u_\h(x,x,p',\mu) =\frac{1}{(2\pi \h)^n} \int e^{-i \hbar \xi_1 \xi_2} c(x,\mu,p';~\h~;~ \xi) d\xi,
\]
where $\xi=\left( \xi_{1},\xi_{2} \right)$ are the dual variables to $(s,\omega)$, and 
\[
c ( x,\mu,p';~ \h~;~ \xi)=\frac{1}{2\pi} \int e^{-(s,\omega) \cdot \xi}a(s,x,x,(-\omega+\mu,p'),\hbar^{-1} \mu) \rho(\mu) ds d \omega
\]
is the Fourier transform of $\rho a$ in $(s,\omega)$.   Note that
\[
\int c(x,\mu , p';~ \h;~ \xi ) d \xi = 2\pi~ a(0,x,x,(\mu,p'),\h^{-1}\mu),
\]
and therefore (by (\ref{also-well-known}))
\begin{equation}\label{faltaba_numero}
u_\h(x,x,p',\mu) - \frac{1}{(2\pi \h)^n}~2\pi~ a(0,x,x,(\mu,p'),\h^{-1}\mu)= \frac{1}{(2\pi \h)^n} \int R( \xi, \hbar ) c( \xi )d \xi,
\end{equation}
where $\left| R(\xi,\hbar) \right| \le \hbar \left| \xi_{1} \xi_{2} \right|$.  Integrating (\ref{faltaba_numero}) we see that the error term in (\ref{lemaDos}) is bounded by
\begin{equation}\label{remedios}
\frac{1}{(2\pi \h)^n} \int_K \norm{Rc}_{L^1_\xi}(x,p',\mu)\, dx\, dp'\, d\mu ,
\end{equation}
where $K$ is a compact set containing the support of the left-hand side of (\ref{faltaba_numero}). Using (\ref{well-known-ineq}) again,
\[
\hbar \left| \left| \xi_{1} \xi_{2} c \right| \right|_{L^{1}_\xi} = \frac{\hbar}{2\pi} \left| \left| \int e^{-i (s,\omega) \cdot \xi} \frac{\partial^2}{\partial s \partial \omega} a(s,x,x,(-\omega+\mu,p'),\hbar^{-1} \mu) \rho(\mu ) ds d \omega \right| \right|_{L^{1}}
\]
\[
\le \frac{\hbar}{2\pi} \sum_{|\alpha| \le 3} \left| \left| \partial^{\alpha} \frac{\partial^{2}}{\partial s \partial \omega} a(s,x,x,(-\omega+\mu,p'),\hbar^{-1} \mu) \rho(\mu)  \right| \right|_{L^{1}_{(s,\omega)}}.
\]
Since $a(s,x,y,p,\sigma)$ is a classical symbol in $\sigma$ of order $m'$ and compactly supported in the rest of the variables, $\forall \alpha$ there exists a constant $C=C(\alpha)$  such that
\begin{equation}
\label{ineq_mu}
\begin{aligned}
\left| \frac{\hbar}{2\pi} \partial^{\alpha} \frac{\partial^2}{\partial s \partial \omega}a(s,x,x,(-\omega+\mu,p'),\hbar^{-1} \mu) \rho(\mu) \right| \le C(\alpha)\left( 1+\hbar^{-1}|\mu| \right)^{m'} \hbar \\
=C(\alpha) \left(\hbar +|\mu| \right)^{m'} \hbar^{-m'+1}
\end{aligned}
\end{equation}
for all $(s,x,y,p)$.  Integrating (\ref{ineq_mu}) with respect to $(s, \omega)$ over a sufficiently large compact set we obtain that
\[
\norm{Rc}_{L^1_\xi}(x,p',\mu)\leq C \left(\hbar +|\mu| \right)^{m'} \hbar^{-m'+1}.
\]
We now integrate this inequality over the compact set, $K$, in (\ref{remedios}), to obtain that the error term in (\ref{lemaDos}) is bounded above by a constant times
\begin{equation}\label{error_bound}
\hbar^{-m'-n+1}\int_{|\mu|\leq \mu_1}\left(\hbar +|\mu| \right)^{m'} d\mu
\end{equation}
for some $\mu_1>0$ independent of $\h$, and this is $O (\hbar^{-m'-n+1})$ when $m' \ge 0$.
\end{proof}

\begin{remark}\label{NotaError}
Lemma \ref{lemmaTwo} implies the Theorem (with a better error estimate in case $m'=0$) if the amplitude $a$ is homogeneous in the variable $\sigma$.
\end{remark}

\medskip
\noindent
{\em Proof of Theorem \ref{Trace1}.}
Let us first assume that $m' > 0$. Since $a$ is a symbol in $\sigma$ of degree $m' > 0$, there exists $\tilde{a}(x,p)$ and a constant $C$ such that
\[
\left| a(0,x,x,p,\sigma)-\sigma^{m'} \tilde{a}(x,p) \chi(\sigma)\right| \le C~(1+|\sigma|)^{m'-1},
\]
where $\chi (\sigma)$ is smooth in $\sigma \neq 0$, and homogeneous of degree zero in $\sigma$. Then, in particular,
\[
\left| \frac{2\pi}{(2\pi \h)^n}~a(0,x,x,\mu,p',\hbar^{-1}\mu)-\frac{2\pi}{(2\pi \h)^n} ~\hbar^{-m'} \mu^{m'}\chi(\mu) \tilde{a}(0,x,(\mu,p')) \right| 
\]
\[
\le 
(2\pi)^{1-n} C~(1+\hbar^{-1}|\mu|)^{m'-1} \h^{-n}.
\]
The left-hand side of this inequality is supported in $\mu \in [-\mu_0, \mu_0]$. After integrating with respect to $\mu$, the remainder is bounded by constant times
\[
\h^{-n} \int_{-\mu_0}^{\mu_0}(1+\hbar^{-1}|\mu|)^{m'-1}d\mu=2\hbar^{-n-m'+1}\int_{0}^{\mu_0}(\hbar+\mu)^{m'-1}d\mu 
\]
\[ 
= 2\hbar^{-n-m'+1} \left( (\mu_0+\hbar)^{m'}-\h^{m'} \right)/m'=O(\hbar^{-n-m'+1}) \text{ since }m' > 0.
\]
Therefore, for any $\ell'$ and $m' > 0$
\[
\int u_\h(x,x)\,dx = (2\pi)^{-n} \hbar^{-\ell-m-n} \int 2\pi~p_{1}^{m'} \tilde{a}(0,x,p) dx dp +O(\hbar^{-\ell-m-n+1}),
\]
where $2\pi ~ p_{1}^{m'} \tilde{a}(0,x,p)$ is the principal symbol on the diagonal.

Now let's assume $m'=0$. Since $a$ is a symbol in $\sigma$ of degree zero, there exists $\tilde{a}(x,p)$ such that,
\[
\left| a(0,x,x,p,\sigma)-\tilde{a}(x,p)\chi(\sigma) \right| < \frac{C}{1+\left| \sigma \right|}
\]
for some constant $C >0$.Then
\[
\left| \frac{2\pi}{(2\pi \h)^n} a(0,x,x,\mu,p',\hbar^{-1}\mu)-\frac{2\pi}{(2\pi \h)^n} \tilde{a}(x,(\mu,p'))\chi(\mu) \right| <  \frac{(2\pi)^{1-n} C}{1+\hbar^{-1}\left| \mu \right|}~\h^{-n}
\]

Again, since the left hand side is supported in the set $\left\{ | \mu | \le \mu_0 \right\}$, after integrating with respect to $\mu$, the remainder is bounded by a constant times
\[
\h^{-n} \int_{0}^{\mu_0}\frac{1}{1+\hbar^{-1}\mu}d\mu=\hbar^{-n+1} \int_{0}^{\mu_0}\frac{1}{\mu+\hbar} d\mu
\]
\[
= \hbar^{-n+1} \left( log(\mu_0+\hbar)-log(\hbar) \right)
=O(\hbar^{-n+1} ~ log(1/\hbar)).
\]
Therefore, for $m=1/2$
\[
\int u_\h(x,x)\,dx =  (2\pi)^{-n} \hbar^{-\ell-m-n} \int 2\pi~ \tilde{a}(0,x,p) dpdx +O (\hbar^{-\ell-m-n+1} ~\log(1/\hbar)),
\]
where $ 2\pi~ \tilde{a}(0,x,p)$ is the principal symbol of the family on the diagonal.

\subsection{The trace in case $m' \le -4$.}

\begin{theorem} With the previous notation, if $m\le -7/2$, and $\widehat A$ an operator in ${\mathcal A}_{X_c}$, 
then each classical $\Psi$DO $\sigma_1(A)_s$ obtained from the symbol of $\widehat A$ in the flow out on each orbit is of trace class, and
\begin{equation}\label{} 
\tr(\widehat A) = (2\pi)^{-n+1/2} \h^{-n-\ell+1/2} \int_{ S}\tr \left( \sigma_1(A)_s\right) ds +O(\hbar^{-n-\ell+3/2}).
\end{equation}
\end{theorem}

\begin{lemma}
If $m'< 0$ and $\ell'=0$,
\begin{equation}\label{lemaQ}
\int u_\h(x,x)\, dx =\frac{1}{(2\pi \h)^n} 2\pi\int a(0,x,x,\mu, p', \h^{-1}\mu)\, dx\,dp'\,d\mu + O(\h^{-n+2}).
\end{equation}
\end{lemma}
\begin{proof}
This follows immediately from (\ref{error_bound}) (which was derived under no assumptions on $m'$).
\end{proof}
\begin{proof}
Starting with equation (\ref{resfrio}), since $m' \le -4$
\[
a_{2}(s,x,y,p) := \int e^{is\sigma}a(s,x,y,p,\sigma)d\sigma
\]
is absolutely convergent and can be extended to a $C^{2}$ function of $s$. In addition, $a_2$ is compactly supported in $s,x,y,p$. Using the stationary phase theorem for $C^{2k}$ amplitudes (Theorem 7.7.5 in \cite{Ho}), we get
\[
\int e^{i\hbar^{-1}sp_1}a_2(s,x,x,-p_1,p')dsdp_1\sim 2\pi\, \hbar\, a_2(0,x,x,0,p'),
\]
and then
\[
\int u_\h(x,x)\,dx = (2\pi)^{-n+1/2} \hbar^{-n-\ell+1/2} \int \sqrt{2 \pi}~ a_2(0,x,x,0,p')dxdp'+O(\hbar^{-n-\ell+3/2}),
\]
where $\sqrt{2 \pi} ~a_2(0,x,x,0,p')=\sqrt{2 \pi} \int e^{is\sigma}a(s,x,x,p)d\sigma_{\big|_{s=0,p_1=0}}$ is the extension of the symbol $\sigma_1$ to the intersection of the Lagrangians.
This is the desired result in the model case.
\end{proof}

%-------section-------

\section{Projectors and ``cut" quantum observables}

\label{sec:Projector}

In this section we will prove that, under a mild additional condition on $\partial X_c$, the algebra ${\mathcal A}_{X_c}$ 
contains orthogonal projectors.  We will also prove that, 
in case there exists an $\h$-pseudodifferential operator on $M$, $\hat{P}$, such that:
\begin{enumerate}
\item  The spectrum of $\hat{P}$ is discrete and is contained in $\h\bbZ$, and
\item $X_c = P^{-1}(I)$ for $I\subset\bbR$ a closed interval,
\end{enumerate}
then the spectral projector of $\hat{P}$ associated to the interval $I$ is in  ${\mathcal A}_{X_c}$ 

%----subsection------

\subsection{On the Existence of Projectors}\label{ProjectorExistence}

In addition to the assumptions on $\partial X_c$ that we have been making throughout, let us now assume 
that $\partial X_c$ is of contact type.  Recall that this means that there exists a one-form $\beta$ on 
$\partial X_c$ such that (a) $d\beta$ is the pull-back of the symplectic form to $\partial X_c$, and 
(b) $\beta\rfloor \Xi_P$ is constant, where $P$ is a defining function of $X_c$ with periodic flow on $\partial X_c$.

Following the proof of Lemma 5 in  \cite{Do}, one obtains:

\begin{lemma}
There exists an smooth function, which will be called again $P:T^*M\to \mathbb{R}$, such that
\begin{itemize}
\item[(a)]
P is bounded from below and tends to $\infty$ at infinity in the cotangent directions,
\item[(b)]
$\partial X_c = P^{-1}(0)$, and
\item[(c)]
There exists a neighborhood $\mathcal W$ of $\partial X_c$ such that the Hamilton flow of $P$ is $2\pi$-periodic in $\mathcal W$.
\end{itemize}
\end{lemma}

Next we recall (see \cite{HR} Proposition 3.8) how to obtain a quantum version of the previous result:

\begin{lemma}
Let $\widehat P(\h)$ be a semiclassical pseudodifferential operator with principal symbol $P$ and vanishing sub-principal symbol. 
Let $\mu$ be the Maslov index of the trajectories of $\Phi^P$ (the Hamilton flow of $P$)
in $\mathcal W$. Assume the Bohr-Summerfeld conditions \eqref{Eq:BohrSomm}. 
Then there exists a semiclassical pseudodifferential operator $\widehat R_2 (\h)$ of order $-2$ such that for $\epsilon <<1$
\[
\text{Spec} \left (P-\frac{\mu}{4}~\h-\widehat R_2 (\h) \right) \cap [-\epsilon/3,\epsilon/3] \subset \h \mathbb{Z}
\]
when we restrict $\h$ to the sequence $\h=1/N$ with $N$ large.
\end{lemma}

\begin{proof}  Pick $\epsilon>0$ such that $P^{-1}[-\epsilon, \epsilon]\subset \mathcal W$.
Let $\rho$ be a smooth function with support in $[-\epsilon,\epsilon]$, such that $\rho\equiv 1$ on $[-\epsilon/ 2, \epsilon/2]$. 
Let
\[
\gamma=\frac{1}{2\pi} \int_0^{2\pi}p(t)\dot{x}(t)-P(x(t),p(t)) dt
\]
be the (common) action of the trajectories of the Hamilton flow of $P$ in $\mathcal W$.  Then
$e^{-2\pi i \h^{-1} \left( \widehat P- \frac{\mu}{4}~\h-\gamma \right)}$ is microlocally in $\mathcal W$
a pseudodifferential operator with symbol identically equal to one, and thus one can write
\[
\rho(\widehat P) e^{-2\pi i \h^{-1} \left( \widehat P- \frac{\mu}{4}~\h-\gamma \right)}=\rho(\widehat P) (I+\h \widehat R(\h)),
\]
where $\widehat R(\h)$ is a zeroth order $\h$-$\Psi$DO.  
Recall that $\gamma$ is an integer in $\partial X_c$, and therefore, for $\h=1/N$ we obtain
\[
\rho(\widehat P) e^{-2\pi i \h^{-1} \left( \widehat P- \frac{\mu}{4}~\h \right)}=\rho(\widehat P) (I+\h \widehat R(\h)),
\]
Since $I+\h \widehat R(\h)\rho(\widehat P)$ has spectrum close to $1$ for $\h<<1$, one can then define for $\h$ small
\[
\widehat R_2=-\frac{\h}{2\pi i}\log(I+\h \widehat R(\h)\rho(\widehat P)),
\]
and since $\widehat R_2$ commutes with $\widehat P$, we obtain
\[
\rho(\widehat P)e^{-2\pi i\h^{-1}\left( \widehat P-\frac{\mu}{4}~\h-\widehat R_2 \right) }=
\rho(\widehat P)\left( I+\h \widehat R(\h) \right) \left( I+\h \widehat R(\h) \rho(\widehat P) \right)^{-1}.
\]
Since $\rho\equiv 1$ on $[-\epsilon/2,\epsilon/2]$, the spectral theorem guarantees that the above operator is the identity on any eigenfunction
of $\hat P$ with eigenvalue in $[-\epsilon/2,\epsilon/2]$.   Since, for $\h$ small enough, the spectrum of 
$\widehat P-\frac{\mu}{4}~\h-\widehat R_2 $ in $[-\epsilon/3, \epsilon/3]$ corresponds to eigenfunctions of $\hat P$ with
eigenvalues in $[-\epsilon/2,\epsilon/2]$, the result follows.
\end{proof}

Let us define $\widehat P_2:=\widehat P- \frac{\mu}{4}~\h-\widehat R_2$, whose principal symbol continues to be $P$. 
Let $\chi$ be the characteristic function on $(-\infty, 0 ]$, and define the projector
\[
\Pi=\chi \left( \widehat P_2 \right)
\]

\begin{theorem}
\label{thm:ProjInAlg}
The projector $\Pi$ defined above belongs to the class $J^{-1/2,1/2}(M\times M, \Delta, \F)$.
\end{theorem}

\begin{proof}  Let $\rho$ be the cut-off function of the previous proof.
We decompose
\[
\Pi=(1-\rho)\chi \left( \widehat P_2 \right) +\rho \chi \left( \widehat P_2 \right)
\]

Clearly $(1-\rho)\chi \left( \widehat P_2 \right)$ is a semiclassical pseudodifferential operator. 
Therefore we need to prove that $\rho ~\chi \left( \widehat P_2 \right)$ belongs to $J^{-1/2,1/2}$.
This operator has microsupport in $\mathcal W$. 

We take $\mathbb{R}^{n-1}\times S^1$ as the model case with coordinates $(x,\theta)$ and $T^{*}(\bbR^{n-1}\times S^1)$ with coordinates $(x,\theta;p,\tau)$. Let $\Pi^E=\Pi_{N}^E$ be the 
projector on eigenfunctions of $\widehat P^n:=\h D_\theta =\frac{\h}{i}\frac{\partial}{\partial \theta}$ 
with eigenvalues greater than or equal to $E/N$, where $E\in \mathbb{Z}$ is a constant. 
(Here $\theta$ is the $2\pi$-periodic variable in $S^1$.)
Let $T_s=e^{-i s \h^{-1} \widehat P_2}$, and let 
\[
T_s^n=e^{-i \h^{-1} s ~\widehat P^n }
\]
be the translation representation on $L^2(\mathbb{R}^{n-1}\times S^1)$. Let $(x_0,p_0)\in \mathcal W$. 
As in \cite{Gui, GuLe}, there exist an $S^1$-invariant neighborhood $\mathcal U\subset \mathcal W$  of 
$(x_0,p_0)$ (the circle action given by
the Hamilton flow of $P$), an $S^1$-invariant open set $\mathcal U^n\subset T^*(\mathbb{R}^{n-1}\times S^1)$, and an $S^1$-equivariant canonical transformation
\[
\phi:\mathcal U\to \mathcal U^n,
\]
which sends $\partial X_c \cap {\mathcal U}$ into 
$\left\{ (x,\theta ; p , \tau) \in T^*\left( \mathbb{R}^{n-1} \times S^1 \right)~ \big |~ \tau=E \right\}$. 
Again, as in \cite{Gui, GuLe}, one can show that there exists a semiclassical zeroth order Fourier integral operator 
\[
F:L^2(M)\to L^2(\mathbb{R}^{n-1}\times S^1),
\]
with microsupport on $\mathcal U\times \mathcal U^n$ such that
\[
F^* F=I_{\mathcal U^n},~~F F^*=I_{\mathcal U},
\]
and
\[
F \rho~\chi (\widehat P_2)=\rho \left( \Pi^E \right) F,
\]
This reduces the proof to the model case.   It suffices to show that $\Pi \widehat Q$ is on the algebra, for any zeroth order 
compactly supported semiclassical pseudodifferential operator $\widehat Q$ in $\mathbb{R}^{n-1} \times S^1$. 
Note that
\[
 \Pi^E~\widehat Q=\frac{1}{2\pi} \int_0^{2\pi} e^{-i N s \widehat P^n } e^{i N s E}~ \widehat Q~\frac{1}{1-e^{is}} ds
\] 
The operator $e^{-i N s \widehat P^n } e^{i N s E} \widehat Q$ is a semiclassical FIO with Lagrangian
\[
\left\{ \left( (x,\theta ~;~p, \tau=E) , (y=x, \alpha+s = \theta~;~ -p, -\tau=-E) \right) \big | x \in \mathbb{R}^{n-1}, \theta, \alpha \in [0, 2\pi] \right\}.
\]
Therefore, in local coordinates, the Schwartz kernel of $\Pi^E \widehat Q$ can be written as
\begin{equation}
\label{KernelPiQ}
\frac{1}{(2\pi \h)^n} \frac{1}{2\pi} \int e^{i N \left( ~(x-y) p+(\theta-\alpha-s) (\tau-E) ~\right)}q(x,\theta,y,\alpha,p,\tau,\h) \frac{1}{1-e^{is}} dp d\tau ds,
\end{equation}
where $q$ is symbol with expansion in $\h$. Notice that $\frac{1}{1-e^{is}}$ is a 
conormal distribution in $s=0$ and equation \eqref{KernelPiQ} shows the hybrid nature of the amplitude of the projector. 
Equation (\ref{KernelPiQ}) also proves that $\Pi^E \widehat Q \in J^{-1/2,1/2}$, with principal symbol $\chi_{X_c}Q(x,\theta~;~p,\tau)$ 
in the diagonal, and
\[
\frac{1}{\sqrt{2\pi}} \frac{Q(\Phi_s^P(x,\theta~;~p,\tau))}{1-e^{is}}
\]
in the flow-out. 
\end{proof}

\begin{remark}If one has an $\h$-pseudodifferential operator $\widehat P$ with discrete spectrum such that 
\[
{\text{Spec}}({\widehat P(\h)})\subset \h\bbZ,
\]
then the previous proof shows that, for any given $E_1,E_2\in \mathbb{Z}$ such that $E_1<E_2$, and for $j=1,2$:
\[
\text{the trajectories}\text{ on } P^{-1}(E_j)\text{ satisfy the Bohr-Sommerfeld condition \eqref{Eq:BohrSomm}.}
\]
Then the orthogonal projector onto 
\[
\calH_N = \mbox{span of eigenvectors of }\widehat{P}(\h)\ \mbox{with eigenvalues in } [E_1 , E_2],
\]
is in the algebra $J^{-1/2,1/2}\left( M\times M ; \Delta, \F \right)$, associated to 
\[
X_c=\left\{ \x \in T^{*}M ~\big | ~ E_1 \le P(\x) \le E_2 \right\}.
\]
\end{remark}

%------subsection--------

\subsection{Cut quantum observables}

In this section we fix a projector $\Pi$ as in \S 5.1, and consider ``cut" quantum observables, by which we mean
operators of the form
\[
\Pi \widehat Q \Pi
\]
where $Q$ is a pseudodifferential operator on $M$.  By Theorem \ref{thm:ProjInAlg} these operators are in $J^{-1/2,1/2}\left( M\times M ;\Delta, \F \right)$. The symbolic properties of these operators are summarized by the following Proposition:

\begin{proposition}\label{Projector}  Let $\widehat{Q}(\h)$ be a zeroth-order semiclassical pseudodifferential operator with compact 
microsupport.  Then  $\Pi \widehat{Q} \Pi $ is in the class $J^{-1/2, 1/2}(M\times M~;~ \Delta, \mathcal{F} \partial X_c)$. Its symbols, 
ignoring Maslov factors, are as follows:
\begin{equation}\label{sigma_0}
\begin{array}{llllll}
\sigma_0  (\Pi \widehat Q \Pi)(\x,\x) = \chi_{X_c} (\x) Q(\x)~\sqrt{dx \wedge dp}, \\ \\
\end{array}
\end{equation}
\begin{equation}\label{sigmaOne}
\begin{array}{llllll}
\sigma_1 \left( \Pi \widehat Q \Pi\right)_s= \Pi_{F_s} M_{|_{Q_{F_s}}} \Pi_{F_s}  \\ \\
\end{array}
\end{equation}
where $\chi_{X_c}$ is the characteristic function of $X_c$,
$\x=(x,p)\in T^{*}M$, $F_s$ is the fiber above $s\in S$, $Q_{|_{F_s}}$ is the restriction of $Q$ to $F_s$, 
$M_{Q_{|_{F_s}}}$ is the operator ``multiplication by $Q_{|_{F_s}}$'',  and $\Pi_{F_s}$ is the Szeg\"o projector in the orbit $F_s$, i.e., for 
$u:F_s\to \mathbb{C}$ smooth, $u(\Phi_s^P \bar y)=\sum_j u_j(\bar y) \frac{e^{ijs}}{\sqrt{2\pi}}$, 
$[ \Pi_{F_s}u ](\bar x)=\sum_{j\ge 0} \frac{u_r(\x)}{\sqrt{2\pi}}$.
\end{proposition}

\begin{remark}
Notice that the Szeg\"o projector $\Pi_{F_s}$ is a classical pseudodifferential operator with principal symbol $\chi_{(T^* F_s)^+}$, where 
$(T^* F_s)^+$ is the part with positive momentum variable, in the direction of the Hamilton flow. 
This function is smooth in $T^*F_s \setminus 0$. 
\end{remark}

\begin{remark}
The symbol of $\Pi_{|_{F_s}} M_{Q_{|_{F_s}}} \Pi_{|_{F_s}}$ is $\chi_{(T^* F_s)^+} Q_{|_{F_s}}$, which agrees with the symbol in the diagonal, restricted to the intersection. This is the so-called symbolic compatibility condition referred to in Theorem \ref{Thm:SymbComp}.
\end{remark}

\begin{proof}
The first part was proven in Theorem \ref{thm:ProjInAlg}. The principal symbol in the diagonal is clear. Using the relation (0.2) in \cite{AUh}, we obtain that for $\y \in \partial X_c$, $s\neq 0$,
\[
\sigma_{1}(\Phi_{s}^P (\y)\ ;\ \y)=\frac{1}{\sqrt{2\pi}}\frac{1}{2\pi} \int_0^{2\pi}\frac{Q(\Phi_{s-\tilde s}^P (\y))}{1-e^{i(s-\tilde s)}}\frac{1}{1-e^{i \tilde s}}d\tilde s
\]
\[
=\frac{1}{\sqrt{2\pi}}\frac{1}{2\pi} \int_{0}^{2\pi}\frac{Q(\Phi_{\tilde s}^P(\y))}{1-e^{i\tilde s}} ~\frac{1}{1-e^{i (s-\tilde s)}}d \tilde s .
\]
Since $Q(\Phi_{\tilde s}^P(\y))$ is a smooth $2\pi$-periodic function in $\tilde s$, there exists a sequence of functions $\left\{ Q_{j}(\y) \right\}_{j=-\infty}^{\infty}$ such that
\[
Q(\Phi_{\tilde s}^P(\y))=\sum_{j=-\infty}^{\infty}Q_{j}(\y) \frac{e^{ij \tilde s}}{\sqrt{2\pi}},
\]
and the symbol becomes:
\[
\sigma_{1}(\Phi_{s}^P (\y)\ ; \ \y ) = \frac{1}{\sqrt{2\pi}}\frac{1}{2\pi} \int \left( \sum_{j}\frac{Q_{j}(\y)}{\sqrt{2\pi}} e^{ij\tilde s} \right) \left( \sum_{k\ge 0}e^{ik \tilde s} \right) \left( \sum_{k' \ge 0}e^{ik' (s-\tilde s)} \right)d \tilde s
\]
\[
=\frac{1}{\sqrt{2\pi}} \frac{1}{1-e^{is}}\left[
\sum_{j \ge 0}e^{ijs}\frac{Q_{j}(\y)}{\sqrt{2\pi}} + \sum_{j<0} \frac{Q_j(\y)}{\sqrt{2\pi}}
\right].
\]

We now interpret this as the kernel of an operator acting on the fibers of $\partial X_c\to S$.
Let us consider a fiber $F_s \subset \partial X_c$, and a function $u:F_s\to \mathbb{C}$.  For every fixed $\y \in \partial X_c$, let $u_r(\y)$ and $Q_j(\y)$ be the corresponding Fourier coefficients in each decomposition 
\[
u(\Phi_s^P \y)=\sum_{r=-\infty}^\infty u_r(\y)\frac{e^{irs}}{\sqrt{2\pi}}, ~~Q(\Phi_s^P \y)=\sum_{j=-\infty}^\infty Q_j(\y)\frac{e^{ijs}}{\sqrt{2\pi}}
\]
As a pseudodifferential operator, the symbol in in the flow-out of $\Pi \widehat Q \Pi$ acting the function $u$, is given by
\begin{equation}
\label{eq:PiQPiSymbOp}
\begin{aligned}
& \left[ \sigma_{1}\left( \Pi \widehat Q \Pi \right)_s ~u \right] (\x)=\int_{F_s}\sigma_1(\Pi \widehat Q \Pi)(\x,\y)u(\y)d\y
 =\int_{S^1}\sigma_1(\Pi \widehat Q \Pi)(\x,\Phi_s^P \x)u(\Phi_s^P \x)ds \\
&=\int \frac{1}{\sqrt{2\pi}} \frac{1}{1-e^{-is}} \left[ \sum_{j\ge 0}\frac{Q_j(\x)}{\sqrt{2\pi}}+\sum_{j<0}\frac{Q_j(\x)}{\sqrt{2\pi}}e^{ijs} \right] \sum_{r} u_r(\x) \frac{e^{irs}}{\sqrt{2\pi}}ds \\
&=\sum_{r\ge 0, j\ge -r}u_{r}(\x) \frac{Q_j(\x)}{\sqrt{2\pi}}.
\end{aligned}
\end{equation}

On the other hand,
\[
\left[ \Pi_{F_s} M_{|_{Q_{F_s}}} \Pi_{F_s} ~(u) \right] (\x)=\Pi_{F_s} M_{|_{Q_{F_s}}}\left( \sum_{r\ge 0}\frac{u_r (\y)}{\sqrt{2\pi}} \right) (\x)=\Pi_{F_s} \left(Q(\y) \sum_{r\ge 0}\frac{u_{r}(\y)}{\sqrt{2\pi}} \right) (\x)
\]
\[
=\Pi_{S^1}\left( Q(\Phi_s^P \x)\sum_{r\ge 0}\frac{u_r (\Phi_s^P \x)}{\sqrt{2\pi}} \right)_{\big |s=0}=\Pi_{S^1}\left( \sum_{j\in \mathbb{Z},r\ge 0} \frac{Q_j(\x)}{\sqrt{2\pi}} u_r (\x) \frac{e^{i(j+r)s}}{\sqrt{2\pi}} \right)_{\big |_{s=0}}
\]
\[
=\sum_{r\ge0, j\ge -r} Q_j(\x) \frac{u_r (\x)}{\sqrt{2\pi}},
\]
which agrees with equation \eqref{eq:PiQPiSymbOp}. This proves
\[
\sigma_{1}\left( \Pi \widehat Q \Pi \right)_s=\Pi_{F_s} M_{|_{Q_{F_s}}} \Pi_{F_s},
\]
which yields (\ref{sigmaOne}) after applying Proposition \ref{prop:SymbCompoModelCase}.

\end{proof}

\medskip
Operators of the form $\Pi_N \widehat{Q}_N \Pi_N:\mathcal{H}_N\to \mathcal{H}_N$ generalize Toeplitz matrices,
 and this is reflected in its principal symbol in $\F$.  (In case $M=S^1$ and $\Pi_N$ the projector onto the span of
 $\{e^{ij\theta},\ j=0,\ldots , N\}$, the $\Pi_N \widehat{Q}_N \Pi_N$ are to leading order
 the generalized Toeplitz matrices of \cite{GS}, page 84.)

%----subsubsection-----

\subsubsection{Applications: A symbolic proof of the Szeg\"o Limit Theorem}
\label{sec:ApplSzego}

We begin with the functional calculus, the heart of which is the following

\begin{lemma}
\label{Lemma_ExpAlgebra}
Let $\widehat Q$ be a self-adjoint pseudodifferential operator of order zero on $M$.  Then
\[
\Pi\,e^{-i t\Pi \widehat Q \Pi} \in J^{-1/2,1/2}\left( M \times M;  \Delta, \mathcal{F} \partial X_c \right) .
\]
\end{lemma}

\begin{proof}
Let us define
\[
W(t)=\Pi e^{-i t\Pi \widehat Q \Pi}.
\]
It is the solution of the problem
\begin{equation}
\label{Eq:Construction}
\left\{
\begin{array}{ll}
\frac{1}{i}\frac{\partial}{\partial t} W(t)+\Pi \widehat Q \Pi W(t)=0 \\ \\
W(t)_{\big |_{t=0}}=\Pi
\end{array}
\right.
\end{equation}
The idea in the following proof is to construct a solution which will be in the algebra and will make the right-hand side of first 
equation \eqref{Eq:Construction} of order $O(\h^\infty)$. 

As a first approximation we take
\[
\widetilde W_0= \Pi e^{-it \widehat Q},
\]
which satisfies
\[
\frac{1}{i}\frac{\partial}{\partial t}\widetilde W_0+\Pi \widehat{Q} \Pi~ \widetilde W_0=-\Pi \left[ \Pi, \widehat{Q} \right] e^{-it \widehat{Q}}.
\]
We will prove below that $\left[ \Pi, \widehat{Q}\right] \in \text{sc-}I^{-1/2}(M \times M;\mathcal{F}\partial X_c)$ 
(see Section \ref{sec:TheCommutator}).  Therefore, we obtain
\[
\left\{
\begin{array}{ll}
\frac{1}{i}\frac{\partial}{\partial t} \widetilde W_0(t)+
\Pi \widehat Q \Pi \widetilde W_0(t)=: \widetilde R_0(t)\in \text{sc-}I^{-1/2}\left(M\times M; \mathcal{F}\partial X_c  \right) \\ \\
{\widetilde W_0(t)}_{\big |_{t=0}}=\Pi
\end{array}
\right.
\]
We will now modify $W_0$ so as to make the right hand side $O(\h^\infty)$ instead of an operator in
$\text{sc-}I^{-1/2}(M\times M; \mathcal{F}\partial X_c)$. 
For the rest of the proof, it will be convenient to identify symbols in the flow-out with corresponding families of smoothing operators acting on 
functions on the fibers $F_s$ for each $s\in S$. The symbol $R_0(t)$ has a corresponding family operator 
$\{\widetilde{ \mathcal R}_{0,s}(t)\}_{s\in S}$. 
Let us consider the following problem,
\[
\left\{
\begin{array}{lll}
\frac{1}{i}\frac{\partial}{\partial t} \widetilde{\mathcal V}_{0,s}+\Pi_{F_s} M_{Q_{|_{F_s}}} \Pi_{F_s} \circ \widetilde{\mathcal V}_{0,s}=
-\widetilde{ \mathcal R}_{0,s}, \\ \\
{\widetilde{\mathcal V}}_{0,s_{\big |_{t=0}} }=0,
\end{array}
\right.
\]
whose solution is
\[
\widetilde{\mathcal V}_{0,s}(t)=-i\int_0^t e^{i(\tilde t-t) \Pi_{F_s} M_{Q_{|_{F_s}}} \Pi_{F_s}} \widetilde{\mathcal R}_{0,s}(\tilde t) d\tilde t.
\]
Notice that $\widetilde{\mathcal V}_{0,s}$ is a smoothing operator. 
Let us call $V_0$ the smooth symbol in $\F$ given by the operator 
$\widetilde{\mathcal V}_{0,s}$. Take $\widetilde{V_0}\in \text{sc-}I^{-1/2}(M\times M;\mathcal{F}\partial X_c)$ with symbol $V_0$. 
By construction
\[
\dfrac{1}{i}\dfrac{\partial}{\partial t}\left( \widetilde{W_0} +\widetilde{V}_0\right)+\Pi \widehat Q  \Pi \left( \widetilde{W_0}+\widetilde{V_0} \right) = \h \widetilde R_1(t) \in \text{sc-} I^{-3/2}\left( M\times M; \mathcal{F}\partial X_c \right).
\]
Proceeding inductively one can find a sequence of operators $\tilde V_j$ such that for all $J$
\begin{equation}
\label{eq:SolFlowOut}
\begin{aligned}
\dfrac{1}{i}\dfrac{\partial}{\partial t}\left( \widetilde{W_0} +
\sum_{j=0}^J \h^{j}~ \widetilde V_j \right)+\Pi \widehat Q  \Pi \left( \widetilde{W_0}+\sum_{j=0}^J \h^{j} ~\widetilde{V_j} \right)  \\
= \h^{J+1} \widetilde R_{J+1} \in \text{sc-}I^{-3/2-J}\left( M\times M; \mathcal{F}\partial X_c \right)
\end{aligned}
\end{equation}

Finally, take an operator $\widetilde V\in \text{sc-}I^{-1/2}(M\times M;\mathcal{F}\partial X_c)$ 
such that $\widetilde V\sim \sum_{j=0}^{\infty}  \widetilde V_{j}$,
and define $\widetilde W = \widetilde{W}_0+\widetilde V$.  Then
\[
\frac{1}{i}\frac{\partial}{\partial t}\widetilde{W}+\Pi \widehat Q \Pi \widetilde{W}=O(\h^\infty).
\]
 A standard application of Duhamel's principle finishes the proof.
\end{proof}

\begin{proposition}
\label{Cor_fPiQPi}
Let $\widehat Q$ be a self-adjoint semiclassical pseudodifferential operator.  Then
for any smooth function $f$, $\Pi\,f(\Pi \widehat Q \Pi)$, is in the class 
$ J^{-1/2,1/2}(M \times M~;~ \Delta, \mathcal{F}\partial X_c)$. The symbols, ignoring Maslov factors, are as follows:
\begin{equation}\label{sigma_0_cor}
\begin{array}{llllll}
\sigma_0 \left( \Pi\,f(\Pi \widehat Q \Pi) \right)(\x,\x)=\chi_{X_c}(\x) f \left(  Q(\x) \right)~\sqrt{dx \wedge dp}, \text{ and}
\end{array}
\end{equation}
\begin{equation}\label{sigmaOne_cor}
\begin{array}{llllll}
\sigma_{1} \left(\Pi\, f (\Pi \widehat Q \Pi) \right)_{\big |_{F_s}}= \Pi_{F_s}\,f\left( \Pi_{F_s} M_{Q_{ |_{F_s}}} \Pi_{F_s} \right),
\end{array}
\end{equation}
where $F_s$,$Q_{|_{F_s}}$, $M_{Q_{ |_{F_s}}}$, and $\Pi_{F_s}$ are as in Proposition \ref{Projector}
\end{proposition}

\begin{proof}
We have:
\begin{equation}
\label{eq:fPiQPi}
\Pi\,f\left( \Pi \widehat Q \Pi \right)=\frac{1}{\sqrt{2\pi}} \int \Pi\, e^{-it \Pi \widehat Q \Pi} \check{f}(t)dt,
\end{equation}
where $\check{f}(t)=\frac{1}{\sqrt{2\pi}}\int e^{is}f(s)ds$. 
By the previous lemma we can conclude that $\Pi f(\Pi \widehat Q \Pi) \in J^{-1/2,1/2}$.  Moreover
\[
\sigma_0(\Pi f(\Pi \widehat Q \Pi))(\x, \x)=\frac{1}{\sqrt{2\pi}}\int e^{-it Q(\x)}\chi_{X_c}(\x) \check{f}(t)dt=f(Q(\x))~\chi_{X_c}(\x),
\]
and
\[
\sigma_1\left(\Pi f(\Pi \widehat Q \Pi) \right)_s=\frac{1}{\sqrt{2\pi}}\int \Pi_{F_s} e^{-i t \Pi_{F_s} M_{Q_{F_s}} \Pi_{F_s}}\check{f}(t)dt=\Pi_{F_s} f(\Pi_{F_s} M_{Q_{F_s}} \Pi_{F_s} ) .
\]
\end{proof}

\medskip
As an immediate corollary of Theorems \ref{Trace1} and \ref{Projector}, we obtain the following Szeg\"o limit theorem:
\begin{corollary}\label{cor:SzgoGeneral}
Assume that $X_c$ is compact.   Then for any smooth function $f$
\[
\tr\left( \Pi_N f (\Pi_N\widehat{Q}_N \Pi_N) \right) = (2\pi)^{{-n}}\ N^{n}\,\int_{X_c} f\circ Q\ \frac{\omega^n}{n!} + O(N^{n-1}\log(N)).
\]
\end{corollary}

%-----subsubsection-----

\subsubsection{Commutators}
\label{sec:TheCommutator}

We now describe another property of the projector. Let $\widehat Q$ be a semiclassical pseudodifferential operator as above. 
The projector $\Pi$ behaves microlocally as the identity on the interior $X_c$, suggesting that $\left[ \Pi, \widehat Q \right]$ is microlocally
 $O(\h^\infty)$ on the diagonal. Using Proposition \ref{IntersJs}, we anticipate that 
 $\left[ \Pi, \widehat Q \right] \in \text{sc-}I^{-1/2}\left( M\times M; \mathcal{F}\partial X_c \right)$.  We now prove that this is indeed
 the case, and compute the principal symbol of the commutator.

\begin{proposition}
\label{Prop:CommutatorPiQ}
For any zeroth order compactly supported semiclassical pseudodifferential operator $\widehat Q$, $\left[ \Pi, \widehat Q \right] \in \text{sc-}I^{-1/2}\left( M\times M; \mathcal{F} \partial X_c \right)$ is a semiclassical Fourier integral operator, with (smooth) principal symbol
\begin{equation}
\label{SymbComm}
\sigma\left( \left[ \Pi, \widehat Q \right] \right)\left( \x=\Phi_s^P \y\ ; \ \y \right)=\left\{
\begin{array}{lll}
\frac{1}{\sqrt{2\pi}}\frac{Q(\y)-Q(\x)}{1-e^{is}}& &\text{ if }\x \neq \y \\ \\
\frac{1}{i \sqrt{2\pi}} \left\{ P, Q\right\}(\x)& & \text{ if }\x=\y.
\end{array}
\right.
\end{equation}
Furthermore, if the principal symbol of $\widehat Q$ is constant along the orbits in the flow-out, the Hamilton flows $\Phi^P$ and $\Phi^Q$ of $P$ and $Q$ commute in $\partial X_c$, and the 
subprincipal symbol of $Q$ vanishes on $\partial X_c$ (Levi condition), then
 $\left[ \Pi, \widehat Q \right] \in I^{-5/2}\left( M\times M ; \mathcal{F}\partial X_c \right)$. 
\end{proposition}

\begin{remark}
The Hamilton flow of the principal symbol of an operator that commutes with $\Pi$ preserves the region $X_c$. 
However, the converse is not true; the invariance of the region $X_c$ is a much weaker condition than the commuting property. 
\end{remark}

\begin{proof}
It is enough to prove it in the model case $\bbR^{n-1} \times S^1$ with coordinates $(x,\theta)$ and $T^*(\bbR^{n-1}\times S^1)$ with coordinates $(x,\theta \, ; \, p,\tau)$. We only consider one energy level, say $E$ so that $X_c=\left\{ \tau \ge E\right\}$.  For simplicity, let us consider zeroth order semiclassical pseudodifferential operators of the form:
\[
\widehat Q(\h) f(x,\theta)=\sum_m e^{im \theta}\int q(x,p,\theta,\h m) e^{ixp/\h}\hat f(p,m,\h)dp,\text{ where}
\]
\[
\hat f(p,m,\h)=\frac{1}{(2\pi \h)^{n+1}} \int e^{-iyp/\h} e^{-im\alpha} f(y,\alpha)dyd\alpha,
\]
and $q(x,p,\theta,s)$ is the full symbol.  Decompose $q$ in its Fourier modes,
$
q_k(x,p,\theta,s)=\sum e^{ik\theta} q_k(x,p,s).
$
Then $\widehat Q=\sum \widehat Q_k$, where 
\begin{equation}
\label{Qk}
\widehat Q_k(\h) f=\sum_m e^{im\theta} \int e^{ik\theta}q_k(x,p,\h m)e^{ixp/\h}\hat{f}(p,m,\h)dp
\end{equation}
is a $\text{sc-}\Psi DO$ with symbol $q_k(x,p,s)$.  
The kernel of $Q_k$ is
\[
K_{Q_k}(x,y,\theta,\alpha)=\sum_m e^{i(k+m)\theta-im\alpha}p_k(x,y,m,\h),
\]
where
\[
p_k(x,y,m,\h)=\frac{1}{(2\pi \h)^{n+1}}\int e^{i(x-y)p/\h}q_k(x,p,\h m) dp.
\]
A calculation shows that for $k>0$,
\[
K_{[\Pi_N,\widehat Q_k(1/N)]}(x,\theta,y,\alpha)
\]
\[
=\frac{N^{n+1}}{(2\pi)^{n+1}}\int e^{iN (x-y)p}e^{iN E (\theta-\alpha)} 
\left[ \sum_{0<j\le k}e^{-i j (\theta-\alpha)}e^{ik \theta}~q_k(x,p,E-j/N)\right] dp .
\]
Notice that the amplitude in the integral above (in brackets) has an expansion in powers of $\h$, and in fact is a semiclassical symbol. 
The phase parametrizes the flow-out of $\left\{ \tau=E \right\}$ by the canonical $S^1$ action, which is $\F$. 
This proves that the commutator is in the corresponding class, and the principal symbol is
\[
\frac{1}{\sqrt{2\pi}} \sum_{0<j\le k}e^{-i j(\theta-\alpha)}e^{i k\theta}q_k(x,p,E)=\frac{q_k(x,p,E)e^{ik\alpha}-q_k(x,p,E)e^{ik\theta}}{1-e^{-i(\alpha-\theta)}}.
\]
The case $k<0$ is similar, and taking the sum over $k$, we obtain \eqref{SymbComm}.

For the last part, assume that the principal symbol is constant in the fibers of $\partial X_c\to S$, which implies that 
$\left[ \Pi, \widehat Q \right] \in \text{sc-}I^{-3/2}\left( M\times M; \F \right)$. 
Assuming the Levi condition, we will show next that the principal symbol (corresponding to the degree $-3/2$) vanishes again. 
Take $\x_0, \y_0\in \F$, $\x_0 \neq \y_0$. 
Consider two zeroth order semiclassical pseudodifferential operators 
$\widehat T_1$, $\widehat T_2$ of disjoint compact microsupport, such that their principal symbol is $1$ in a 
neighborhood $\x$, $\y$ respectively. 
Notice that
\begin{equation}
\label{eq:LocComm}
\widehat T_1 \left[ \Pi, \widehat Q \right] \widehat T_2 = \left[ \widehat T_1 \Pi \widehat T_2, \widehat Q \right] +
\widehat T_1 \Pi \left[ \widehat Q, \widehat T_2 \right]+\left[ \widehat Q, \widehat T_1 \right] \Pi \widehat T_2.
\end{equation}
Near $(\x,\y)$, the symbol of $\widehat T_1 \left[ \Pi, \widehat Q \right] \widehat T_2$ and $\left[ \Pi, \widehat Q \right]$ coincide. 
Consider, on the other hand, the first term on the right-hand side of equation \eqref{eq:LocComm}.  First, by the assumption on the
microsupports of $T_1$ and $T_2$, the operator $ \widehat T_1 \Pi \widehat T_2$ does not have wave-front set along the diagonal,
and therefore it is in $\text{sc-}I^{-1/2}\left( M\times M; \F \right)$.  We can then apply Proposition \ref{prop:QVComm} below
to compute the symbol of the commutator $\left[ \widehat T_1 \Pi \widehat T_2, \widehat Q \right]$.
Near $(\x_0,\y_0)$ the symbol $ \widehat T_1 \Pi \widehat T_2$ is equal to the symbol of $\Pi$, and
clearly the (diagonal) Lie derivative of this symbol with respect to the Hamilton flow of $Q$ is zero.   Therefore
the symbol of this commutator (as an operator of order $-3/2$) is zero.

The principal symbols of the last two terms in \eqref{eq:LocComm}
also vanish because the principal symbols of $T_1,T_2$ are constant near $\x_0,\y_0$, respectively. 
Therefore, the $(-3/2)$ principal symbol of $\left[ \Pi, \widehat Q \right]$ vanishes off the diagonal, and therefore 
everywhere on $\F$ by continuity.
This concludes the proof.
\end{proof}

%---------section----------

\section{On some propagators $e^{-it\h^{-1} \Pi \widehat Q \Pi}$}
\label{sec:Propagator}

\subsection{The classical counterpart}

 We begin by considering classical hamiltonians $Q:T^*M\to\bbR$ with the property that their Hamilton field is tangent to $\partial X_c$,
 that is,   $\Xi_Q (\x) \in T_{\x} \partial X_c$ for all $\x\in \partial X_c$.
 It is easy to see that this occurs if and only if the Poisson bracket satisfies $\{P,Q\}_{|_{X_c}}=0$,  and a as consequence, $Q$ is constant on the fibers of $\partial X_c\to S$.
The Hamilton flow of such a $Q$ preserves the region $X_c$, that is, it defines a classical flow in the symplectic
 manifold with boundary $X_c$.   The restriction of $Q$ to $\partial X_c$ descends to a smooth function $Q_S:S\to\bbR$, which in 
 turn defines a Hamilton flow on $S$.  However,
 the following diagram (where $\Phi^Q_t$ ~denotes the Hamilton flow of $Q$, etc.) does {\em not} commute in general:
 \begin{equation}\label{diagramita}
 \begin{array}{ccc}
 \partial X_c & \overset{\Phi^Q_t|_{\partial X_c}}\longrightarrow & \partial X_c\\
 \downarrow & & \downarrow\\
  S& \overset{\Phi^{Q_S}_t}  \longrightarrow & S
 \end{array}
 \end{equation}

\begin{lemma}
\label{Lemma:QHypothesis}
Assume that $\Xi_Q$ is tangent to $\partial X_c$, and let $P$ be a defining function of $X_c$, as above. 
Then the diagram \eqref{diagramita} commutes for all $t$ if and only if 
the Poisson bracket of $P$ and $Q$ vanishes to second order at the boundary $\partial X_c$, 
by which we mean that there exists a smooth function $F$ such that
\[
 \{P,Q\} = P^{2}F.
\]

\end{lemma}

\begin{proof}
Since $\Xi_Q(\x) \in T_{\x}\partial X_c$ for all $\x \in \partial X_c=P^{-1}(0)$, then
$\left\{ Q,P \right\}(\x)=dP_{\x}(\Xi_Q)=0\text{ for all } \x \in \partial X_c$.
Therefore we can write $\{P,Q\}=F_{0} P$ for some smooth function $F_{0}$, and
\[
[\Xi_P, \Xi_Q]=\Xi_{\{ P, Q\}}=F_{0} \Xi_P+P \Xi_{F_{0}}
\]
will vanish on $\partial X_c$ if and only if $F_{0}$ itself vanishes on $\partial X_c$. 
As a result, 
\[
\Phi_s^P\circ \Phi_t^Q (\x) =\Phi_t^Q \circ \Phi_s^P (\x)
\]
 $\forall \x\in\partial X_c$ if and only if $\{ P, Q \}$ vanishes to second order on $\partial X_c$.
\end{proof}

\medskip
We now discuss the relation of the above considerations with Lerman's  symplectic cut construction.
The ``cut" space is
\[
Y = X_c/\sim,
\]
where the equivalence relation on $X_c$ is:  $\x\sim\y$ iff $\x$ and $\y$ are on the boundary $\partial X_c$ and in fact on
the same leaf of $\partial X_c\to S$.  There is an obvious inclusion $S\hookrightarrow Y$ and it is clear that, as sets,
\begin{equation}\label{pedazos}
Y = \text{Int }(X_c)\coprod S
\end{equation}
(disjoint union).  Let us give $Y$ the quotient topology.  Then  a function $Q$ satisfying $\{Q, P\}|_{\partial X_c} = 0$
induces a continuous function  $Q_Y: Y\to\bbR$, which is smooth when restricted to each of the pieces in \eqref{pedazos}.
If the Poisson bracket vanishes to second order, then, by the previous lemma, one has a commutative diagram
\begin{equation}\label{diagrama}
\begin{array}{ccc}
X_c & \overset{\Phi^Q_t|_{X_c}}\longrightarrow & X_c\\
 \downarrow & & \downarrow\\
Y& \overset{\Phi^{Q_Y}_t}  \longrightarrow & Y
 \end{array}
\end{equation}
where $\Phi^{Q_Y}$ is a Hamilton flow defined piece-wise by restricting $Q$ to the pieces in \eqref{pedazos}.

In Lerman's construction the topological space $Y$ acquires the structure of a symplectic manifold
of which $S$ is a symplectic submanifold.  
However, in general, $Q_Y$ {\em is not smooth} with respect to Lerman's structure; for this it is necessary that $\{Q, P\}|_{\partial X_c} = 0$ 
to infinite order, as implied by the following lemma:

\begin{lemma}
If $\{P, Q\}|_{\partial X_c} = 0$ to order $k$, meaning that there exists a smooth function $F$ such that
$\{P, Q\} = P^{k}F $,
then $Q_Y\in C^{k-1}(Y)$.
\end{lemma}
\begin{proof}
Take the model case $M=\bbR^{n-1}\times S^1$ with coordinates $(x,\theta)$, $T^{*}M$ with
canonical coordinates $(x,\theta \, ; \,p,\tau)$, and 
$X_c=\left\{ (x,\theta\,;\, p,\tau) \big | \tau \geq 0\right\}$.  Then the symplectic cut is the manifold 
$Y\cong \bbR^{2(n-1)}_{(x,p)}\times \bbC_{z}$ (where the symplectic form on $\bbC$ is
$(\frac{1}{\sqrt{2}i}dz\wedge d\bar z)$), and the projection $X_{c}\to Y$ is
\[
(x,\theta; p, \tau) \mapsto (x, p ; \sqrt{\tau} e^{i \theta}).
\]
Any function $Q(x,\theta; p,\tau)$ whose Hamilton flow preserves $X_{c}$ descends to the continuous function $Q_{Y}$ on $Y$
\[
Q_Y(x,p,z)=Q(x,\arg z\,;\,p, |z|^2)
\]
for $z\not=0$ and $Q_Y(x,p,0) = Q(x,\theta\,;\,p, 0)$ for any value of $\theta$.
(The continuity of $Q_{Y}$ will be seen in what follows.)
If $\left\{ P, Q\right\}$ vanishes to order $k$ in $\partial X_c$, then
\[
\left\{ P, Q\right\}=\frac{\partial Q}{\partial \theta}=\tau^k \, F(x,\theta\,;\,p,\tau)
\]
for some smooth function $F$.  This implies that there exist smooth functions $G_1(x,\theta\,;\,p,\tau),\  G_2(x\,;p,\tau)$ such that
\[
Q =  \tau^k \, G_{1}(x,\theta \, ; \, p,\tau) + G_2(x\, ; \, p,\tau).
\]
It follows that
\[
Q_Y(x,p,z)=\begin{cases}
|z|^{2k}\, G_1(x, \arg z\, ;\, p, |z|^{2})+G_2(x\, ; \, p, |z|^{2})\quad \text{if}\ z\not= 0\\ \\
G_{2}(x\, ; \, p,0) \quad \text{if}\ z= 0 .
\end{cases}
\]
Define $\tilde{Q}_{Y} (x',\xi', z) = Q_Y(x',\xi',z) - G_2(x', |z|^{2},\xi')$.  Then $Q_{Y}$ and $\tilde{Q}_{Y} $ differ by a smooth function,
and therefore it suffices to show that $\tilde{Q}_{Y} (x',\xi', z)$ is $C^{k}$.  For this, we'll show that a 
function of the form
\[
\tilde G^k (x,p,z)=\begin{cases}
|z|^{2k}\, G(x,\arg z \, ; \, p, |z|^{2})\quad \text{if}\ z\not= 0\\ \\
0 \quad \text{if}\ z= 0
\end{cases}
\]
is a $C^{k}$ function for any smooth function $G(x,\theta \,; \, p,\tau)$.  The function $\tilde G^k$ is smooth in the region $z\neq 0$, so
we only need to show the existence and continuity of partial derivatives at $z=0$. 
We will prove the statement by induction. Write $z=(u_1,u_2)$. For $k=1$,
\[
\frac{\partial \tilde G^1(x,p,z)}{\partial u_1}_{\big |_{z=0}}=\lim_{u_1\to 0}\frac{u_1^2 G(x,0\,;\, p, u_1^2)}{u_1}=0,
\]
and for $z\not= 0$
\[
\frac{\partial \tilde G^1}{\partial u_1}=2u_1 G(x,\arg z \, ; \, p, |z|^{2})-u_2 \frac{\partial G}{\partial \theta}(x,\arg z \, ; \, p, |z|^{2})+2 u_1|z|^2 \frac{\partial G}{\partial \tau}(x,\arg z \, ; \, p, |z|^{2}),
\]
which converges to zero as $z\to 0$. The partial derivative with respect to $u_2$ is similar, showing that $\tilde G^1 \in C^1$.

Assume that the statement is valid for $k-1>0$. Notice that
\[
\frac{\partial \tilde G^k}{\partial u_1}=2k u_1 \tilde G^{k-1}-u_2\, \widetilde{\left( \frac{\partial G}{\partial \theta} \right)}^{k-1}+2u_1\widetilde{ \left( \frac{\partial G}{\partial \tau} \right)}^k, \text{ and }
\]
\[
\frac{\partial \tilde G^k}{\partial u_2}=2k u_2\tilde G^{k-1}+u_1\widetilde{\left( \frac{\partial G}{\partial \theta} \right)}^{k-1}+2u_2\widetilde{\left( \frac{\partial G}{\partial \tau} \right)}^k.
\]
 Each of the terms on the right-hand side of each of the equalities above are at least $C^{k-1}$, finishing the proof.

\end{proof}

%-----subsection----

\subsection{A Symbolic Description of The Propagator $e^{-it~h^{-1} \Pi \widehat{Q} \Pi }$}

Throughout this section $Q$ will denote a smooth function such that the Poisson bracket $\{P ,Q\}$ vanishes to second
order at $\partial X_c$ (c.f.~  Lemma \ref{Lemma:QHypothesis}).   As we saw in the previous section we then obtain
a classical flow $\Phi_t^Q|_{X_c}$ that descends to a continuous flow on the cut space $Y$.  In this section we analyze the quantum mechanical propagator  $e^{-it~h^{-1} \Pi \widehat{Q} \Pi }$, where $ \widehat{Q}$ is a 
semiclassical pseudodifferential operator on $M$ with symbol $Q$
and whose subprincipal symbol satisfies 
\[
\text{Sub } \widehat Q _{\big |_{X_c}}=0 .
\]
Before we state the main result, let us start with a proposition:

%---proposition---

\begin{proposition}
\label{prop:QVComm}
For each semiclassical Fourier integral operator $\widetilde V(\h)$ in \\ 
$\text{sc-}I^{-1/2}\left( M\times M; \F \right)$, 
the commutator $\left[ \widehat Q(\h), \widetilde V(\h) \right]$ is in the class\\
$\text{sc-}I^{-3/2}\left(M\times M; \F \right)$.
Its principal symbol is
\[
\sigma_{[\widehat Q,\widetilde V]}(\bar x,\bar y)=\frac{\h}{i} \mathfrak{L}_{\Xi_Q}V(\bar x,\bar y),
\]
where $V$ is the principal symbol of $\widetilde V$, and $\mathfrak{L}_{\Xi_Q}$ is the Lie derivative obtained by letting
the Hamilton flow of $Q$ act diagonally on $\F$.

\end{proposition}

\begin{proof}
Write the Schwartz kernel of $\widehat Q$ in the model case as
\[
\widehat Q(y,x)=\frac{1}{(2\pi \h)^n} \int e^{i(y-x)p/h}q(y,p,\h) dp,
\]
where 
\[
q(y,p,\h)\sim q_0(y,p)+\h q_1(y,p)+\ldots,
\]
and the Schwartz kernel of $\widetilde V \in \text{sc-}I^{-1/2}(M\times M~;~\F)$ as
\[
\widetilde V(z,y)=\frac{1}{(2\pi\h)^n}\int e^{i(z'-y')\omega/\h}v(z,y,\omega,\h)d\omega,
\]
where
\[
v(z,y,\omega,\h)\sim v_0(z,y,\omega)+\h v_1(z,y,\omega)+\ldots
\]
Then the Schwartz kernel of $\widetilde V \circ \widehat Q$ is
\[
\begin{aligned}
\widetilde V \circ \widehat Q(z,x)=\frac{1}{(2\pi \h)^n} & \int e^{i\h^{-1}(z'-x')\omega}\\
& \left[ \frac{1}{(2\pi\h)^n}\int e^{i\h^{-1} \left[ (y-x)p+(x'-y')\omega \right]}v(z,y,\omega,\h)q(y,p,\h)dydp \right] d\omega
\end{aligned}
\]
Applying the stationary phase method to the integral in brackets, one gets
\begin{equation}
\label{Eq:CommCont1}
\begin{aligned}
 &\frac{1}{(2\pi\h)^n}\int e^{i\h^{-1}\left[ (y-x)p+(x'-y')\omega \right]}v(z,y,\omega,\h)q(y,p,\h)dydp \\
& \sim v_0(z,x,\omega)q_0(x,(0,\omega))-\frac{\h}{i}\frac{\partial v_0(z,x,\omega)}{\partial x} \frac{\partial q_0(x,p)}{\partial p}_{\big |_{p=(0,\omega)}}
-\frac{\h}{i}v_0(z,x,\omega)\frac{\partial q_0(x,p)}{\partial x\partial p}_{\big |_{p=(0,\omega)}} \\
& +\h v_1(z,x,\omega)q_0(x,(0,\omega)) +\h v_0(z,x,\omega)q_1(x,(0,\omega)).
\end{aligned}
\end{equation}

The Schwartz kernel of $\widehat Q \circ \widetilde V$ can be written as
\[
\begin{aligned}
\widehat Q\circ \widetilde V(z,x)=\frac{1}{(2\pi\h)^n} & \int e^{i(z'-x')\omega/\h} \\
& \left[ \frac{1}{(2\pi\h)^n} \int e^{i\h^{-1} 
\left[ (z-y)p+(y'-z')\omega \right]}q(z,p,\h)v(y,x,\omega,\h)dpdy\right] d\omega
\end{aligned}
\]
Applying stationary phase to the amplitude above, one gets
\begin{equation}
\label{Eq:CommCont2}
\begin{aligned}
& \frac{1}{(2\pi\h)^n} \int e^{i\h^{-1} \left[ (z-y)p+(y'-z')\omega \right]}q(z,p,\h)v(y,x,\omega,\h)dpdy\\
& \sim q_0(z,(0,\omega))v_0(z,x,\omega)+\frac{\h}{i} \frac{\partial q_0(z,p)}{\partial p}_{\big |_{p=(0,\omega)}}\frac{\partial v_0(z,x,\omega)}{\partial z} \\
& +\h q_1(z,(0,\omega))v_0(z,x,\omega)+\h q_0(z,(0,\omega))v_1(z,x,\omega)
\end{aligned}
\end{equation}

Therefore, the commutator $[\widehat Q,\widetilde V]$ 
has as leading amplitude
\[
v_0(z,x,\omega)\left( q_0(z,(0,\omega) -q_0(x,(0,\omega)) \right).
\]
This vanishes at $x'=z'$, which corresponds to the flow-out. 
Therefore $[\widehat Q,\widetilde V]\in \text{sc-}I^{-3/2}(M\times M~;~\F)$. 
In order to compute the principal symbol there, we notice that
\[
q_0(z,(0,\omega)) -q_0(x,(0,\omega))=(z'-x')\cdot d(z,x,\omega),
\]
where $d = (d_2,\ldots ,d_{n})$ is a vector-valued function such that
\[
d(z,x,\omega)_{\big |_{x'=z'}}=\nabla_{ x'}q_0(x,(0,\omega))_{\big |_{x'=z'}}.
\]
Therefore
\[
\frac{1}{(2\pi\h)^n} \int e^{i(z'-x')\omega/\h}\left( q_0(z,(0,\omega) -q_0(x,(0,\omega))
 \right) v_0(z,x,\omega) d\omega
\] 
\begin{equation}
\label{Eq:CommCont3}
=\frac{-\h}{i(2\pi\h)^n} \int e^{i(z'-x')\omega/\h}\sum_{j=1}^{n-1} \frac{\partial}{\partial \omega_j} 
\left( d_j(z,x,\omega) v_0(z,x,\omega) \right) d\omega .
\end{equation}

The principal symbol of $[\widehat Q,\widetilde V]$ in $\text{sc-}I^{-3/2}(M\times M~;~\F)$ can be then computed taking all the contributions from equations \eqref{Eq:CommCont1}, \eqref{Eq:CommCont2} and \eqref{Eq:CommCont3}. Since $q_0(x,p)$ is constant along the orbits, then
\[
\frac{\partial^2 q_0(x,p)}{\partial x_j\partial p_j}_{\big |_{p=(0,\omega)}}=\frac{\partial^2 q_0(z,p)}{\partial x_j\partial p_j}_{\big |_{\stackrel{x'=z'}{p=(0,\omega)} }}, \text{ for any }j\ge 2.
\]
Since $\left\{ Q,P \right\}$ vanishes at second order on $\partial X_c$, then 
$\frac{\partial^2 q_0(x,p)}{\partial x_1 \partial x_1}_{\big |_{p_1=0}}=0$.  
Assuming all these conditions, and the fact that the subprincipal symbol
\[
\text{Sub }\widehat Q(x,p)=q_1(x,p)-\frac{1}{2i}\sum_{j=1}^n \frac{\partial^2 }{\partial x_j \partial p_1}_{\big |_{p=(0,\omega)}}
\]
vanishes on $\partial X_c$, the principal symbol of the commutator reduces to
\[
\frac{\partial q_0(z,p)}{\partial p}_{\big |_{\stackrel{x'=z'}{p=(0,\omega)}}} \frac{\partial v_0(z,x,\omega)}{\partial z}_{\big |_{x'=z'}}
+\frac{\partial q_0(x,p)}{\partial p}_{\big |_{\stackrel{x'=z'}{p=(0,\omega)}}}\frac{\partial v_0(z,x,\omega)}{\partial x}_{\big |_{x'=z'}}
\]
\[
-\sum_{j=2}^n \frac{\partial q_0(x,p)}{\partial x_j}_{\big |_{\stackrel{x'=z'}{p=(0,\omega)}}} \frac{\partial v_0(z,x,\omega)}{\partial \omega}_{\big |_{x'=z'}}.
\]
One can check that this is the Lie derivative of $V$ with respect to the Hamiltonian field $\Xi_Q$ in the sense in the statement
of the lemma.
\end{proof}

Notice that there is a way to define classes $J^{\ell,m}$ for a general pair of admissible Lagrangian submanifolds that intersect cleanly (see \cite{MU,GU}), and not only  for the diagonal and the flow-out of $\partial X_{c}$.  The main result of this section is the following:

\begin{theorem}
\label{thm:Propagator}
Suppose $\widehat Q$ is a zeroth-order  semiclassical pseudodifferential operator satisfying the conditions of Lemma \ref{Lemma:QHypothesis}. Assume $sub \widehat Q(\h)=0$. Then
\[
\Pi e^{-it\h^{-1} \Pi \widehat Q \Pi}  \in J^{-1/2,1/2}\left( M\times M ; \Delta(t), \F (t) \right),
\]
where 
\begin{equation}
\label{tDelta}
\Delta(t)=\left\{ \left( \x,\y \right) \big | \x, \y \in T^{*}(M\times M), \x=\Phi_t^Q (\y) \right\}
\end{equation}
\begin{equation}
\label{tLambda}
\F(t)=\left\{ \left( \x,\y \right) \big | \x,\y \in \partial X_c,~ \exists s \in \mathbb{R} \textrm{ such that } \x=\Phi_s^P \Phi_t^Q (\y) \right\}.
\end{equation}

\end{theorem}

\begin{remark}
In this statement $t$ is a parameter, but we could also consider $t$ as a variable (in which case the
kernel of the operator would be a family of functions on $\bbR\times M\times M$).  
Also, the symbols of $\Pi e^{-it\h^{-1} \Pi \widehat Q \Pi} $
can easily be computed.
\end{remark}

\begin{proof}
Let us define the following operator:
\begin{equation}
\label{eq:W}
W(t):=\Pi e^{-it \h^{-1} \Pi \widehat Q \Pi} \, e^{it\h^{-1} \widehat Q}.
\end{equation}

We first prove the following:
\begin{lemma}
$W(t)\in J^{-1/2,1/2}\left( M\times M;\Delta, \F \right)$, and the principal symbol in the diagonal is $\sigma_0 =\chi_{X_c}$.
\end{lemma}

\begin{proof}
Let us define $D_t=\frac{1}{i}\frac{\partial}{\partial t}$. $W(t)$ satisfies the following equation
\begin{equation}
\label{Eq_WPi}
\left\{
\begin{array}{lll}
\h D_t W(t)+\left[ \widehat Q, W(t) \right]+\left[ \Pi, \widehat Q \right] W(t)=0\\ \\
W_{\big |_{t=0}}=\Pi
\end{array}
\right.
\end{equation}

Similarly to the proof of Lemma \ref{Lemma_ExpAlgebra}, using the symbol calculus we will construct a sequence of approximate
solutions of equation \eqref{Eq_WPi}. 
This construction makes the right-hand side of order $O(\h^\infty)$, and an application of Duhamel's principle concludes the proof. 

As a  first approximation we take $\widetilde W_{0}=\Pi$.  This is a sensible choice since
\[
\h^2 \widetilde S_2:= \h D_t \widetilde W_{0}+\left[ \widehat Q, \widetilde W_{0} \right]+\left[ \Pi, \widehat Q \right] \widetilde W_{0}=-\left[ \Pi, \widehat Q \right] (I-\Pi) \in \text{sc-}I^{-5/2}\left( M\times M;\F \right),
\]
by Proposition \ref{Prop:CommutatorPiQ}.

We now modify $\widetilde W_{0}$ by elements in  $I(M\times M; \F)$ to lower the order of the remainder. 
It is easy to see that the symbol of the correction term is the solution to the problem
\[
\left\{
\begin{array}{ll}
\frac{\partial V_1 (\bar x, \bar y, t)}{\partial t}+\mathfrak{L}_{\Xi_Q}V_1(\bar x, \bar y,t)=-i S_2,\\ \\
{V_1}_{\big |_{t=0}}=0
\end{array}
\right.
\]
where $S_1$ is the principal symbol of $\widetilde S_1 \in \text{sc-}I^{-1/2}(M\times M;  \F)$.  
Let $\widetilde V_{1}\in\text{sc-}I^{-1/2}(M\times M;  \F)$ 
be an operator with this as symbol, and let $\widetilde W_{1}=\widetilde W_{0}+\h \widetilde V_{1}$.

Since $\left[ \Pi, \widehat Q \right] \in \text{sc-}I^{-5/2} (M \times M;\F)$ by Proposition \ref{Prop:CommutatorPiQ}, we get
\[
\h^3 \widetilde S_{3}:= \h D_t \widetilde W_{1}+\left[ \widehat Q, \widetilde W_{1} \right]+\left[ \Pi, \widehat Q \right]  \widetilde W_{1} 
 \in I^{-7/2}(M\times M ; \F)
\]

Proceeding inductively in this fashion, we obtain an infinite sequence $\{\widetilde V_{j}\}$ such that for all $J$
\[
\h D_t \left( \widetilde W_{0} +\sum_{j=1}^J \h^j \widetilde V_{j} \right)+\left[ \widehat Q,  \widetilde W_{0}+\sum_{j=1}^J \h^j \widetilde V_{j} \right]
+\left[ \Pi, \widehat Q \right] \left( \widetilde W_{0}+\sum_{j=1}^J \h^j \widetilde V_{j} \right)
\]
\[
= \h^{J+2} \widetilde S_{J+2} \in I^{-5/2-J}(M\times M; \F).
\]
Next we take an operator $\widetilde V\in \text{sc-}I^{-1/2}\left(M\times M; \F \right)$ 
such that $\widetilde V\sim \sum_{j=1}^\infty \h^j \widetilde V_j$, and define $\widetilde W=\widetilde W_0+\widetilde V$. 
\end{proof}

Going back to the proof of the theorem, notice that
\[
\Pi e^{-it \h^{-1} \Pi \widehat Q \Pi}=W(t)e^{-it N \widehat Q}.
\]
The Lagrangian $\Delta(t)$ intersects $\Delta$ and $\F$ transversally. 
Using a variation of the Proposition 4.1 in \cite{GU}, we conclude that composing elements in $J^{-1/2,1/2}(M \times M; \Delta, \F)$ 
with $e^{-it \h^{-1} \widehat Q}$ gives elements in $J^{-1/2,1/2}(M\times M; \Delta(t), \F(t) )$. 
\end{proof}

%----subsection-----

\subsection{An Egorov-type Theorem}

We can easily prove the following corollary.

\begin{corollary}
Let $\widehat Q$ be a zeroth order semiclassical pseudodifferential operator satisfying the conditions of Lemma \ref{Lemma:QHypothesis}, and the Levi condition. Then for any zeroth order semiclassical pseudodifferential operator $\widehat A(\h)$, we have 
\[
\widetilde B (t):=e^{it \h^{-1} \Pi \widehat Q \Pi} \Pi \widehat A \Pi e^{-it \h^{-1} \Pi \widehat Q \Pi} \in J^{-1/2,1/2}\left( M\times M,\Delta, \F \right),
\]
with the following principal symbols:
\[
\begin{array}{lll}
\sigma_0 (\widetilde B(t)) (\x,\x)=\chi_{X_c} (\x)~ \left(a \circ \Phi_t^{Q} (\x)\right) & \text{ for } (\x,\x) \in \Delta\setminus\Sigma, and\\ \\
\sigma_1  (\widetilde B(t))_{\big |_{F_s}}=\Pi_{F_s} M_{{(a \circ \Phi_t^Q)}_{|_{F_s}}} \Pi_{F_s},
\end{array}
\]
where $a$ is the principal symbol of $\widehat A$, $\Phi_t^Q$ the Hamilton flow $Q$, 
$F_s$ is an orbit in $\partial X_c$, and $\Pi_{F_s}$ is the Szeg\"o projector of $F_s$.
\end{corollary}

\begin{proof}
Let us consider $W(t)$ as in equation \eqref{eq:W}. 
Notice that
\[
\widetilde B(t)=e^{it\h^{-1} \Pi \widehat Q \Pi} \Pi \widehat A \Pi e^{-i t N \Pi \widehat Q \Pi}=W(-t) e^{it \widehat Q} \widehat A e^{-i t \widehat Q} W(-t)^*,
\]
which proves it belongs to the algebra. Since $e^{it \h^{-1} \widehat Q} \widehat A e^{-i t \h^{-1} \widehat Q }$ is a semiclassical 
pseudodifferential operator with symbol $a \circ \Phi_t^Q$, the symbol on the diagonal can be trivially obtained. 
To compute the symbol on the flow-out, we note that the principal symbols of $W$ are exactly those of $\Pi$ and we
use Proposition \ref{CompoModelCase}:
\[
\sigma_1 (\widetilde B(t))(\y=\Phi_s^P(\x), \x)=
\]
\[
=\frac{1}{\sqrt{2\pi}} \int \sigma_1 \left( W(-t) \right)(\y, \Phi_{\tilde s}^P (\x)) \sigma_1 
\left( e^{it \widehat Q}\widehat A e^{-it \widehat Q} W(-t)^* \right) (\Phi_{\tilde s}^P (\x), \x) d\tilde s
\]
\[
=\frac{1}{\sqrt{2\pi}} \int \frac{1}{\sqrt{2\pi}} \frac{1}{1-e^{i(s-\tilde s)}} a(\Phi_{t}^Q \Phi_{\tilde s}^P (\x)) \frac{1}{\sqrt{2\pi}} \frac{1}{1-e^{i\tilde s}} d\tilde s.
\]
As a pseudodifferential operator on the fibers, this is $\Pi_{F_s} M_{{(a \circ \Phi_t^Q)}_{|_{F_s}}} \Pi_{F_s}$.
\end{proof}

%----------section-----------

%---section----

\section{A numerical study of propagation of coherent states}

\label{sec:SCNA}

Let $\widehat Q$ be a zeroth order semiclassical pseudodifferential operator with symbol $Q$. It is well-known that, if 
$\psi_{(x_0,p_0)}$ is a coherent state with center at $(x_0,p_0)$, then $e^{-it\h^{-1}\widehat Q}(\psi_{(x_0,p_0)})$ is a coherent
state (appropriately  ``squeezed'') with center at $\Phi_{t}^Q(x_0,p_0)$, where $\Phi_{}^Q$ is the Hamilton flow of $Q$.  
If the flow $\Phi_{}^Q$
preserves $X_{c}$ and the center $(x_0,p_0)$ is in the interior of $X_{c}$, then the same conclusion holds for the propagation
$e^{-it\h^{-1}\Pi\widehat Q\Pi}(\psi_{(x_0,p_0)})$ of the coherent state by
$\Pi\widehat Q\Pi$, as the trajectory of the center will remain away from the boundary $\partial X_{c}$ and
everything is as if we were in the boundaryless case.  

In this section we present results of a numerical calculation of $e^{-it\h^{-1}\Pi\widehat Q\Pi}(\psi_{(x_0,p_0)})$
in an example where the Hamilton flow of $Q$ does not preserve $X_{c}$, that is, 
trajectories of $\Phi_{}^Q$ cross the boundary $\partial X_{c}$.

We consider the Harmonic oscillator $\widehat P=\frac{1}{2}(x^2-\h^2 \partial_x^2)$ in $\bbR^1$, and the corresponding projector $\Pi$ onto the 
span of its eigenfunctions of with eigenvalues less than or equal to one.   We take $Q=x^2-p^2$, and $\widehat Q$ the obvious  quantization of $Q$.
Figure \ref{fig:ContourNonComm} shows some energy levels of $Q$. Notice that the energy levels cross the boundary of $X_c=\left\{ x^2+p^2 \le 2 \right\}$.

\begin{figure}[h!]
\begin{center}
\resizebox{1.8 in}{1.8 in}
{\includegraphics[width=1cm]{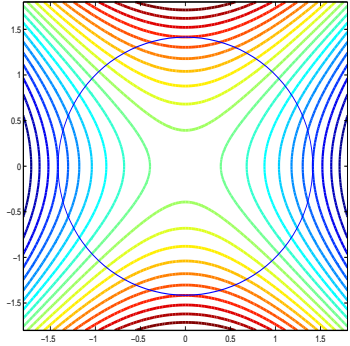}}
\end{center}
\caption{\label{fig:ContourNonComm}{Energy levels of $ Q$ and the boundary $\partial X_{c}=\{P=1\}$.}}
\end{figure}

We now take a coherent state centered inside the interior of $X_c$, and numerically compute its propagation under $\Pi\widehat Q\Pi$.
We do the calculation  in Bargmann space, for symplicity.

We recall that the Bargmann space is defined as the Hilbert space
\[
 \mathcal{B}=\left\{ f:\mathbb{C} \to \mathbb{C}: f \text{ entire and } \int \left| f(z) \right|^2 e^{-\frac{|z|^2}{ \hbar}} dm_z < \infty \right\},
\]
where $dm_z=dx dp$ and $z=\frac{x-i p}{\sqrt{2}}$, with the Hermitian inner product
\[
\langle f, g \rangle =\int f(z) \overline{g(z)} e^{-\frac{z\overline z} {\hbar}}dm_z.
\]
The Harmonic oscillator in Bargmann space is given by
\[
\widehat P=\hbar z \frac{\partial}{\partial z}+\frac{\hbar}{2},
\]
with principal symbol  $P(z,\overline z)=z\overline z$ and eigenbasis
\[
\left\{ b_n=\frac{z^n}{\sqrt{\hbar^n n!}} \right\}_{n\in \mathbb{N}\cup {\left\{ 0 \right\}}},
 ~~\widehat P b_n=\hbar \left(n+\frac{1}{2}\right) b_n=\lambda_n b_n.
\]
The quantization of $Q=x^2-p^2$ in Bargmann space is the operator $\widehat Q =\hbar^2 \frac{\partial^2}{\partial z^2}+z^2 $.  
Applying $\widehat Q$ to the eigenbasis, we get
\[
\widehat Q (b_n) =\hbar \sqrt{n (n-1)}b_{n-2}+\hbar \sqrt{(n+1) (n+2)} b_{n+2},
\]
which gives a  a ``generalized'' Toeplitz matrix for $\Pi_N \widehat Q \Pi_N$ for each positive integer $N$, where $\h = \frac{1}{N}$. 
The (normalized) coherent state in Bargmann space, with center at $w$, is given by the simple formula
\[
\Psi_{w}(z) =e^{\frac{z \overline w}{\hbar}}e^{-\frac{w\overline w}{2\hbar}}.
\]
We apply the propagator $e^{-i t\h^{-1} \Pi_N \widehat Q \Pi_N}$ to the projected coherent state
\[
\Psi_w(z,N) := \Pi_{N}\Psi_{w}(z) =
\sum_{n=0}^{N}\frac{\overline w^n}{\sqrt{n! \hbar^n}}e^{-\frac{\overline w w}{2\hbar}} b_n.
\]

In Bargmann space, we measure the concentration in phase space of any semiclassical family $\psi$ by taking the absolute value of the family
times the square root of the Bargmann weight, namely, by forming the Husimi density:
\[
|\psi |_{H}(z) := |\psi(z)|\,e^{-z\zbar/2\h}.
\] 
We took as initial data a projected coherent state with center at $w=-0.25 -0.6i$, which corresponds to $(x,p)=\sqrt{2}\, (-0.25, 0.6)$. Figure \ref{fig:ProHarmOsc} consists of contour plots of the Husimi density in the $z$ variable
of the initial projected coherent state, its propagation at $t=0.25$ (approximately when the center of the coherent state hits the boundary), and at time $t=0.5$.
We observe that after the time of collision of the center with the boundary, 
the coherent state splits into two localized states with centers on inward trajectories and 
with the same classical energy. Here we took $N=100$.
\begin{figure}[h!]
\begin{center}
\resizebox{5.5 in}{1.8 in}
{\includegraphics[width=1cm]{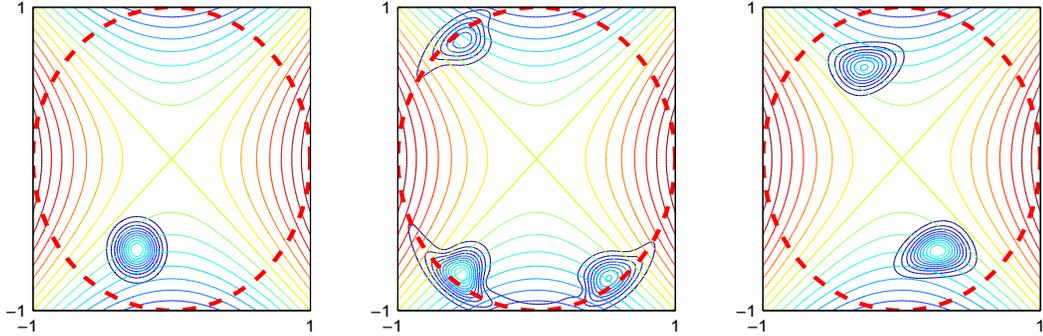}}
\end{center}
\caption{\label{fig:ProHarmOsc} \small{A coherent state is propagated by $e^{-it N\Pi_{N} \hat Q \Pi_{N}}$ ($Q=x^2-p^2$) at time $t=0$ (left), $t=.25$ (middle) and $t=0.5$ (right). The contour plots of the Husimi densities at each time are shown.}}
\end{figure}
The splitting happens immediately after the center collides with the boundary; thus one can speak of infinite-propagation speed along the boundary.
Note that the evolution is time-reversible, so that in some cases the opposite phenomenon will occur, namely, two localized states
with same classical energy will hit the boundary at the same time and combine into one.

\bibliographystyle{plain}
\bibliography{References.bib}

\end{document}